\documentclass[10pt]{article}

\usepackage{amssymb,amsmath,amsfonts,amsthm}
\usepackage[a4paper,includeheadfoot,left=4cm,right=4cm,top=3cm,bottom=3cm]{geometry}
\usepackage{graphicx}
\usepackage[stretch=10]{microtype} 

\usepackage[font={footnotesize}]{caption}

\usepackage[numbers,comma,sort&compress]{natbib}
\makeatletter
\let\@fnsymbol\@arabic
\makeatother

\usepackage{color}
\definecolor{marin}{rgb}{0.,0.3,0.7}
\usepackage[colorlinks,citecolor=marin,linkcolor=marin,urlcolor=marin,bookmarksopen,bookmarksnumbered]{hyperref}

\definecolor{mygreen}{rgb}{0.,0.7,0.3}
\definecolor{myorange}{rgb}{0.9,0.5,0.0}
\definecolor{myred}{rgb}{0.9,0.1,0.2}

\newcommand{\al}{\alpha}

\newcommand{\ga}{\gamma}
\newcommand{\de}{\delta}

\newcommand{\ee}{\epsilon}

\newcommand{\rh}{\rho}
\newcommand{\si}{\sigma}

\newcommand{\om}{\omega}

\newcommand{\Z}{\mathbb{Z}}

\newcommand{\R}{\mathbb{R}}

\newcommand{\T}{\mathbb{T}}

\newcommand{\calC}{\mathcal{C}}
\newcommand{\calE}{\mathcal{E}}
\newcommand{\calK}{\mathcal{K}}
\newcommand{\calO}{\mathcal{O}}

\newcommand{\calS}{\mathcal{S}}
\newcommand{\calU}{\mathcal{U}}
\newcommand{\fd}{\mathbf{d}}
\newcommand{\fe}{\mathbf{e}}
\newcommand{\fg}{\mathbf{g}}
\newcommand{\fk}{\mathbf{k}}
\newcommand{\fr}{\mathbf{r}}
\newcommand{\fu}{\mathbf{u}}
\newcommand{\fv}{\mathbf{v}}
\newcommand{\fz}{\mathbf{z}}

\newcommand{\flambda}{{\boldsymbol{\lambda}}}
\newcommand{\fomega}{\boldsymbol{\omega}}
\newcommand{\fOmega}{\boldsymbol{\Omega}}
\newcommand{\fPhi}{\boldsymbol{\Phi}}
\newcommand{\fPsi}{\boldsymbol{\Psi}}
\providecommand{\abs}[1]{\lvert#1\rvert}
\providecommand{\absbig}[1]{\bigl\lvert#1\bigr\rvert}
\providecommand{\absBig}[1]{\Bigl\lvert#1\Bigr\rvert}
\providecommand{\absbigg}[1]{\biggl\lvert#1\biggr\rvert}

\providecommand{\klabig}[1]{\bigl(#1\bigr)}
\providecommand{\klaBig}[1]{\Bigl(#1\Bigr)}
\providecommand{\klabigg}[1]{\biggl(#1\biggr)}
\providecommand{\norm}[1]{\lVert#1\rVert}
\providecommand{\normbig}[1]{\bigl\lVert#1\bigr\rVert}

\providecommand{\normv}[1]{\ensuremath{{\lVert\hskip-1pt\lvert}#1{\rvert\hskip-1pt\rVert}}}
\providecommand{\normvbig}[1]{\ensuremath{{\bigl\lVert\hskip-1pt\bigl\lvert}#1{\bigr\rvert\hskip-1pt\bigr\rVert}}}

\providecommand{\skla}[1]{\langle#1\rangle}
\providecommand{\ekla}[1]{[#1]}
\providecommand{\eklabig}[1]{\bigl[#1\bigr]}
\newcommand{\formulatext}[1]{\qquad\text{#1}\qquad}
\newcommand\myfor{\formulatext{for}}
\newcommand\myforall{\formulatext{for all}}
\newcommand\myand{\formulatext{and}}

\newcommand\myif{\formulatext{if}}

\newcommand{\iu}{\mathrm{i}}
\newcommand{\e}{\mathrm{e}}
\newcommand{\drm}{\mathrm{d}}

\DeclareMathOperator{\sinc}{sinc}

\newtheorem{theorem}{Theorem}
\newtheorem{lemma}[theorem]{Lemma}
\theoremstyle{definition}
\newtheorem{remark}[theorem]{Remark}
\newtheorem{example}[theorem]{Example}

\newcommand{\jvec}{{\skla{j}}}
\newcommand{\lvec}{{\skla{l}}}

\newcommand{\transpose}{{\hspace{-0.01cm}\scriptscriptstyle\mathrm T}}

\title{Metastable energy strata 
in numerical discretizations 
of weakly nonlinear wave equations
}

\author{Ludwig Gauckler\,\thanks{Institut f\"ur Mathematik,
        TU Berlin,
        Stra{\ss}e des 17.\ Juni 136,
        D-10623 Berlin, Germany. Present address: Institut f\"ur Mathematik,
        FU Berlin,
        Arnimallee 9,
        D-14195 Berlin, Germany ({\tt gauckler@math.fu-berlin.de}).}
        \and
        Daniel Weiss\,\thanks{Institut f\"ur Angewandte und Numerische Mathematik,
        Karlsruher Institut f\"ur Technologie (KIT),
        Englerstr.\ 2,
        D-76131 Karlsruhe, Germany
        ({\tt daniel.weiss@kit.edu}).}
}

\date{Version of 9 December 2016}

\begin{document}

\maketitle

\begin{abstract}
The quadratic nonlinear wave equation on a one-dimensional torus with small initial values located in a single Fourier mode is considered. In this situation, the formation of metastable energy strata has recently been described and their long-time stability has been shown. 
The topic of the present paper is the correct reproduction of these metastable energy strata by a numerical method. For symplectic trigonometric integrators applied to the equation, it is shown that these energy strata are reproduced even on long time intervals in a qualitatively correct way. \\[1.5ex]
\textbf{Mathematics Subject Classification (2010):} 
65P10, 
65P40, 
65M70, 
35L05.\\[1.5ex] 
\textbf{Keywords:} Nonlinear wave equation, trigonometric integrators, long-time behaviour, modulated Fourier expansion.
\end{abstract}

\section{Introduction}

We consider the nonlinear wave equation 
\begin{equation}\label{eq-nlw}
\partial_{tt} u - \partial_{xx} u + \rh u = u^2, \qquad u=u(x,t)\in\R,
\end{equation}
on a one-dimensional torus, $x\in\T=\R/(2\pi\Z)$, with a positive  Klein--Gordon parameter $\rh$. We assume that the initial value consists of a single Fourier mode and is small. 
The smallness of the initial value corresponds to a weakly nonlinear setting. 
In \cite{Gauckler2012}, it has been investigated how the energy, which is initially located in this single Fourier mode, is distributed among the other modes in the course of time. More precisely, the formation of energy strata has been shown that are persist on long time intervals, see Section \ref{sec-main} for a precise description of this result. 

In the present paper, we discretize the nonlinear wave equation and answer the question whether this long-time property of the exact solution is inherited by the numerical method. As a numerical method, we consider a spectral collocation in space combined with a symplectic trigonometric integrator in time. We show that this numerical method in fact reproduces the energy strata of the exact solution even on a long time interval, provided that a certain numerical non-resonance condition is fulfilled, see Section \ref{sec-main} for a formulation of this main result. 

The considered numerical method is already known to behave well on long time intervals with respect to preservation of regularity and near-conservation of actions \cite{Faou2009a,Cohen2008a}, as well as near-conservation of energy and momentum \cite{Cano2006,Cano2013,Cohen2008a,Faou2011}. Our result adds to this list an even more sophisticated long-time property of the exact solution that is reproduced in a qualitatively correct way. In comparison to previous results, the considered situation requires less control of interactions in the nonlinearity, which allows us to exclude numerical resonances under weaker restrictions on the time step-size than needed previously. 

The proof of our main result is given in Sections \ref{sec-mfe}--\ref{sec-proof2}. As for the exact solution, it is based on a modulated Fourier expansion in time, with a multitude of additional difficulties due to the discrete setting that will be described in detail. 

Related questions have been studied for symplectic and non-symplectic methods applied to the nonlinear Schr\"odinger equation. In this case, an initial value consisting of a single Fourier mode yields a solution that continues to consist of a single Fourier mode for all times (plane wave solution). In particular and in contrast to the nonlinear wave equation considered in the present paper, there is no formation of energy strata in the other modes. The stability in numerical discretizations of these plane wave solutions under small perturbations of the initial value is an old question \cite{Weideman1986} and has been studied on short time intervals \cite{Weideman1986,Dahlby2009,Khanamiryan,Lakoba2013,Cano2016} and on long time intervals \cite{Faou2014}.

\section{Statement of the main result}\label{sec-main}

\subsection{Metastable energy strata revisited}\label{subsec-revisited}

Writing the solution $u=u(x,t)$ of the nonlinear wave equation \eqref{eq-nlw} as a Fourier series $\sum_{j\in\Z} u_j(t) \e^{\iu jx}$ with Fourier coefficients $u_j$,
the \emph{mode energies} of the nonlinear wave equation \eqref{eq-nlw} are given by
\begin{equation}\label{eq-modeenergies}
E_j(t) = \tfrac12 \abs{\om_j u_j(t)}^2 + \tfrac12 \abs{\dot{u}_j(t)}^2, \qquad j\in\Z,
\end{equation}
where $\om_j$ are the frequencies
\begin{equation}\label{eq-freq}
\om_j=\sqrt{j^2+\rho}, \qquad j\in\Z.
\end{equation}
Note that $E_j=E_{-j}$ for real-valued initial values (and hence real-valued solutions).

We are interested in the evolution of these mode energies.
Assuming that the initial values are small and concentrated in the first mode,
\[
E_1(0)\le \ee\ll 1, \qquad E_j(0)=0 \myfor \abs{j}\ne 1,
\]
an inspection of interaction of Fourier modes in the nonlinearity $g(u)=u^2$ suggests that the energy in the first mode is distributed among the other modes in a geometrically decaying way:
\[
E_0(t) = \mathcal{O}(\ee^2), \qquad E_j(t) = \mathcal{O}(\ee^{\abs{j}}) \myfor j\ne 0,
\]
at least for small times $t$. 

The main result of \cite{Gauckler2012} (Theorem 1 for $\theta=1$) states that these energy strata are in fact stable on long time intervals in the following sense: for fixed but arbitrary $K\ge 2$ and $s>\frac12$, we have on the long time interval $0\le t\le c \ee^{-K/2}$
\begin{align*}
E_0(t) &\le C \ee^2,\\
E_l(t) &\le C \ee^{l}, \qquad 0<l<K,\\
\sum_{l=K}^\infty \si_l \ee^{-K} E_l(t) &\le C
\end{align*}
with the weights
\begin{equation}\label{eq-sigma}
\si_j = \max\bigl(\abs{j},1\bigr)^{2s}.
\end{equation}
This can be written in compact and equivalent form as
\[
\sum_{l=0}^\infty \si_l \ee^{-e(l)} E_l(t) \le C \myfor 0\le t\le c \ee^{-K/2}
\]
with the energy profile 
\begin{equation}\label{eq-e}
e(j) = \begin{cases} 2, & j=0,\\ \abs{j}, & 0<\abs{j}<K, \\ K, & \abs{j}\ge K, \end{cases}
\end{equation}
see also Figure \ref{fig-e} for an illustration. In addition to these bounds on all mode energies, the first mode energy is nearly conserved on this long time interval, $\abs{E_1(t)-E_1(0)}\le C\ee^2$. All constants depend on $K$ and $s$ but not on $\ee$, and they deteriorate as $K$ or $s$ grow.
This result is valid for all but finitely many values of the parameter $\rho\in[0,\rho_0]$ in the wave equation \eqref{eq-nlw}, which excludes some resonant situations. The question that we want to answer in the present paper is whether this stable behaviour on long time intervals is inherited by a typical structure-preserving numerical discretization of the nonlinear wave equation.

\begin{figure}
\centering






\includegraphics{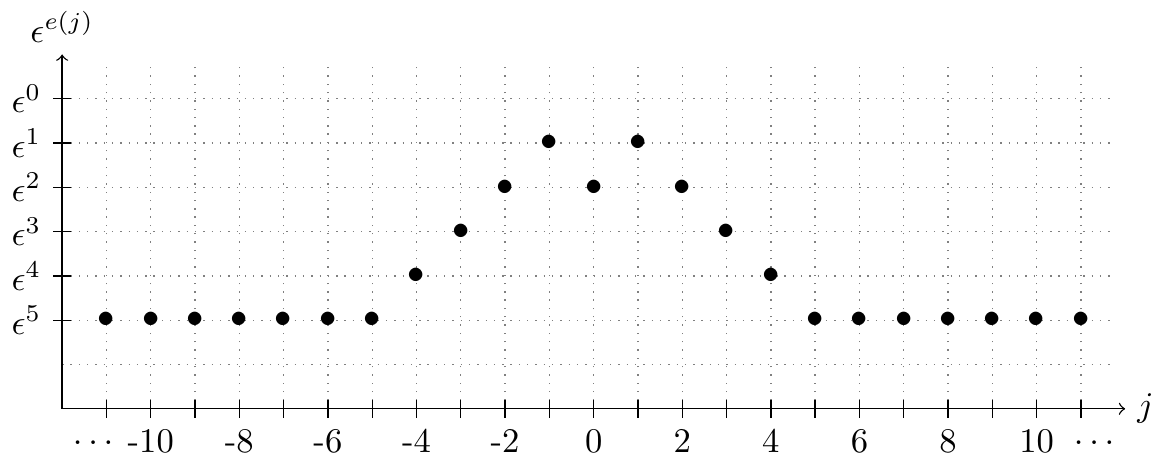}
\caption{Illustration of the bound $\ee^{e(j)}$ for $E_j$ in the case $K=5$.}\label{fig-e}
\end{figure}

\subsection{Trigonometric integrators}

For the numerical discretization of the nonlinear wave equation~\eqref{eq-nlw} we consider trigonometric integrators in time applied to a spectral collocation in space, see, for example, \cite{Bao2012,Cano2006,Cano2013,Cohen2008a,Dong2014,Faou2011,Faou2009a,Gauckler2015}. 

In these methods, the solution $u=u(x,t)$ of \eqref{eq-nlw} is approximated, at discrete times $t_n=n\tau$ with the time step-size $\tau$, by a trigonometric polynomial of degree $M$:
\[
u(x,t_n) \approx \sum_{j=-M}^{M-1} u_j^n \e^{\iu j x}, \qquad n=0,1,\ldots.
\]
The Fourier coefficients $\fu^n=(u_{-M}^n,\dots,u_{M-1}^n)^\transpose$ of this trigonometric polynomials are computed with the trigonometric integrator
\begin{subequations}\label{eq-method}
\begin{equation}\label{eq-method-main}
\fu^{n+1} - 2\cos(\tau\fOmega) \fu^n + \fu^{n-1} = \tau^2 \fPsi \bigl( (\fPhi \fu^n) \ast (\fPhi \fu^n) \bigr),
\end{equation}
where $\fOmega$ denotes the diagonal matrix containing the frequencies $\om_j$, $j=-M,\dots,M-1$, of \eqref{eq-freq}. In addition, $\ast$ denotes the discrete convolution defined by
\[
(\fu \ast \fv)_j = \sum_{j_1+j_2\equiv j\bmod{2M}} u_{j_1} v_{j_2}, \qquad j=-M,\dots,M-1,
\]
which can be computed efficiently using the fast Fourier transform.
The method is characterized by the diagonal filter matrices $\fPsi=\psi(\tau\fOmega)$ and $\fPhi=\phi(\tau\fOmega)$, which are computed from filter functions $\psi$ and $\phi$. These filter functions are assumed to be real-valued, bounded and even with $\psi(0)=\phi(0)=1$. The starting approximation is computed by
\begin{equation}\label{eq-method-starting}
\fu^1 = \cos(\tau\fOmega) \fu^0 + \tau\sinc(\tau\fOmega) \dot{\fu}^0 + \tfrac12 \tau^2 \fPsi \bigl( (\fPhi \fu^0) \ast (\fPhi \fu^0) \bigr),
\end{equation}
and a velocity approximation by
\begin{equation}\label{eq-method-velocity}
2\tau\sinc(\tau\fOmega) \dot{\fu}^n = \fu^{n+1} - \fu^{n-1}.
\end{equation}
\end{subequations}
An error analysis of these methods has been given in \cite{Gauckler2015}.
We assume here that the considered trigonometric integrator is symplectic,
\begin{equation}\label{eq-symplectic}
\psi(\xi) = \sinc(\xi) \phi(\xi),
\end{equation}
see \cite[Chap.~XIII, Eq.~(2.10)]{Hairer2006}. Examples for such filter functions are $\phi=1$ and $\psi=\sinc$, which is the impulse method \cite{Grubmueller1991,Tuckerman1992} or method of Deuflhard \cite{Deuflhard1979}, as well as $\phi=\sinc$ and $\psi=\sinc^2$, which is the mollified impulse method of Garc\'{i}a-Archilla, Sanz-Serna \& Skeel \cite{GarciaArchilla1999}. Certain choices of filter functions leading to non-symplectic methods could also be handled using the transformation indicated in \cite[Remark 3.2]{Cohen2015}, but we do not pursue this here.

We are interested in the mode energies \eqref{eq-modeenergies} along the numerical solution \eqref{eq-method}. We denote them in the following by
\begin{equation}\label{eq-modeenergies-disc}
E_j^n = \tfrac12 \abs{\om_j u_j^n}^2 + \tfrac12 \abs{\dot{u}_j^n}^2, \qquad j=-M,\dots,M-1.
\end{equation}
We note that $E_j^n=E_{-j}^n$ for real-valued initial values (and hence numerical solutions that take real values in the collocation points $x_k=\pi k/M$, $k=-M,\dots,M-1$), and we set $E_{M}^n=E_{-M}^n$.
As for the exact solution in Section \ref{subsec-revisited}, we consider initial values with
\begin{equation}\label{eq-init}
E_1^0 \le \ee, \qquad E_j^0 = 0 \myfor \abs{j}\ne 1.
\end{equation}

\subsection{Metastable energy strata in trigonometric integrators}\label{subsec-result}

We now state our main result which says, roughly speaking, that trigonometric integrators \eqref{eq-method} integrate the metastable energy strata in nonlinear wave equations \eqref{eq-nlw} qualitatively correctly.

For this result, we need a non-resonance condition on the time step-size $\tau$ and on the frequencies $\om_j$, $j=-M,\dots,M-1$, of \eqref{eq-freq}. In the statement of this condition, we consider indices $j\in\{-M,\dots,M-1\}$ and vectors $\fk=(k_0,\dots,k_M)^\transpose$ of integers $k_l$, and we write 
\[
\fk\cdot\fomega = \sum_{l=0}^M k_l\omega_l
\]
with the vector $\fomega=(\omega_0,\dots,\omega_M)^\transpose$ of frequencies.
We fix $3\le K\le M$, and we denote, corresponding to \cite{Gauckler2012}, by $\calK$ the set
\begin{align}
\calK &= \bigl\{\, (j,\fk) : \max(\abs{j},\mu(\fk))<2K \text{ and } k_l=0 \text{ for all } l\ge K \,\bigr\} \label{eq-calK}\\
 &\qquad \cup \bigl\{ (j,\pm \skla{(j-r) \bmod{2M}}+\fk) : \abs{(j-r)\bmod{2M}}\ge K, \, \abs{r}<K,\, \mu(\fk)<K \, \bigr\}, \notag
\end{align}
where $\jvec = (0,\dots,0,1,0,\dots,0)^\transpose$ is the $\abs{j}$th unit vector and 
\begin{equation}\label{eq-mu}
\mu(\fk) = \sum_{l=0}^M \abs{k_l} e(l) = 2\abs{k_0} + \sum_{l=1}^{K-1} \abs{k_l} l + K \sum_{l=K}^M \abs{k_l}.
\end{equation}
In comparison to the set $\calK$ for the exact solution (see \cite[Equation (11)]{Gauckler2012}), we have to consider indices modulo $2M$ due to the discretization in space. In addition, we correct here a typo in \cite[Equation (11)]{Gauckler2012}, where $\abs{j}\ge K$ should be replaced by $\abs{j-r}\ge K$. 

\medskip

\noindent\textbf{Non-resonance condition.} For given $0\le\nu<\frac12$, we assume that there exists a constant $0<\gamma\le 1$ such that 
\begin{subequations}\label{eq-nonres-gesamt}
\begin{equation}\label{eq-nonres-weak}
\absbig{ \sin\bigl(\tfrac12 \tau (\omega_j\pm \fk\cdot\fomega) \bigr) } \ge \gamma \tau \ee^{\nu/2} \myforall (j,\fk)\in\calK, \, \fk\ne\mp\jvec
\end{equation}
and
\begin{equation}\label{eq-nonres-strong}
\absbig{ \sin\bigl(\tfrac12 \tau (\omega_j\pm \fk\cdot\fomega) \bigr) } \ge \gamma \tau \myforall (j,\fk)\in\calK, \, \fk\ne\jvec, \, \fk\ne-\jvec, \, \abs{j}\le K.
\end{equation}
We emphasize that only $\fk=-\jvec$ is excluded in the estimate for $\sin(\tfrac12 \tau (\omega_j + \fk\cdot\fomega) )$ in the first condition \eqref{eq-nonres-weak}, whereas $\fk=\jvec$ and $\fk=-\jvec$ are both excluded in the estimate for this sine  in the second condition \eqref{eq-nonres-strong}.
\end{subequations}

\medskip

Under this non-resonance condition, we prove in this paper the following discrete analogon of the analytical result \cite[Theorem 1]{Gauckler2012} described in Section \ref{subsec-revisited}.

\begin{theorem}\label{thm-main}
Fix an integer $3\le K\le M$ and real numbers $0\le\nu<\frac12$ and $s>\frac12$, and assume that the non-resonance condition \eqref{eq-nonres-gesamt} holds for this $\nu$. Then there exist $\ee_0>0$ and positive constants $c$ and $C$ such that the mode energies \eqref{eq-modeenergies-disc} along trigonometric integrators \eqref{eq-method} for nonlinear wave equations \eqref{eq-nlw} with initial data \eqref{eq-init} with $0<\ee\le \ee_0$ satisfy, over long times
\[
0\le t_n = n\tau \le c\ee^{-K(1-2\nu)/2}, 
\]
the bounds
\[
\sum_{l=0}^{M} \si_l \ee^{-e(l)} E_{l}^n \le C
\]
and $\ee^{-e(1)} \abs{E_{1}^n-E_{1}^{0}}\le C \ee^{1-2\nu}$. 
The constants $c$ and $C$ depend on $K$, $\nu$ and $s$, but not on $\ee$ and the discretization parameters $\tau$ and~$M$.
\end{theorem}

\begin{remark}
The result of Theorem \ref{thm-main} gets stronger for larger $K$ (with worse constants, however), but also the assumption gets stronger, because the set $\calK$ (see \eqref{eq-calK}) in the non-resonance condition \eqref{eq-nonres-gesamt} grows. 
We can thus also cover the (not so interesting) border case $K=2$ as considered in \cite{Gauckler2012} provided that the non-resonance assumption holds for $K=3$. Without requiring the non-resonance condition with $K=3$, our techniques of proof can be used to show that Theorem \ref{thm-main} holds for $K=2$, if we replace $\ee^{-e(1)} \abs{E_{1}^n-E_{1}^{0}}\le C \ee^{1-2\nu}$ by $\ee^{-e(1)} \abs{E_{1}^n-E_{1}^{0}}\le C \ee^{(1-2\nu)/2}$. The same remark applies to \cite{Gauckler2012}.
\end{remark}

The proof of Theorem~\ref{thm-main} is given in Sections \ref{sec-mfe}--\ref{sec-proof2} below. The structure of the proof is similar to the structure of the corresponding proof for the exact solution given in \cite{Gauckler2012}, with a multitude of additional difficulties due to the discrete setting. 

The main difference of this result in comparison with the corresponding result \cite[Theorem 1]{Gauckler2012} for the exact solution is the required non-resonance condition. The non-resonance condition \eqref{eq-nonres-gesamt} excludes two types of resonances. First, it requires that $\om_j\pm\fk\cdot\fomega$ is bounded away from zero for $(j,\fk)\in\calK$ with $\fk\ne\mp\jvec$. This is the non-resonance condition for the exact solution, which can be shown to hold for all except finitely many values of $\rho$ in a fixed interval, see \cite[Section 3.2]{Gauckler2012}. In addition, the non-resonance condition \eqref{eq-nonres-gesamt} requires also that products $\tau(\om_j\pm\fk\cdot\fomega)$ are bounded away from \emph{nonzero} integer multiples of $2\pi$. If this latter condition is violated, we observe numerical resonances as illustrated in the following section.

Numerical resonances can be typically excluded under some CFL-type step-size restriction on the time step-size. A feature of the considered situation is that they can be excluded under less restrictive assumptions than in previous studies on the long-time behaviour of trigonometric integrators for nonlinear wave equations \cite{Cohen2008a,Faou2009a}. In particular, there are no numerical resonances on the time interval $0\le t \le \ee^{-K/2}$ under the CFL-type step-size restriction
\begin{equation}\label{eq-cfl}
\tau (M+K) \sqrt{1+\rho} \le c < \pi,
\end{equation}
for some constant $c$ (we may take $\nu=0$ here). This follows from the structure of the set $\calK$ of \eqref{eq-calK} which enters the non-resonance condition: this set allows only a single large frequency $\om_l\le M\sqrt{1+\rho}$ in the linear combination $\fk\cdot\fomega$, and it allows to bound the remaining part of $\fk\cdot\fomega$ by $2K\sqrt{1+\rho}$ (using the definition of $\mu$ and the bounds on $\mu$ in $\calK$). Under the above CFL condition \eqref{eq-cfl} and using $\om_j\le M\sqrt{1+\rho}$, the arguments $\tfrac12 \tau (\om_j\pm\fk\cdot\fomega)$ of the sines in the non-resonance condition \eqref{eq-nonres-gesamt} are thus bounded by $c<\pi$, which excludes numerical resonances (but not analytical resonances).

In the situation of \cite{Cohen2008a,Faou2009a}, the stronger CFL-type step-size restriction of the form $\tau M K \le c < \pi$ has to be used to exclude numerical resonances. 

For nonzero $\nu$, the numerical non-resonance condition \eqref{eq-nonres-gesamt} resembles the one used in \cite{Cohen2008a}, see Equations (23) and (24) therein. We note that the reduction (24) there is of no relevance in our context because of the structure of the set $\calK$. We also note that the number of indices $(j,\fk)$ that have to satisfy the stronger condition \eqref{eq-nonres-strong}, which does not appear in \cite{Cohen2008a}, is bounded independently of the spatial discretization parameter $M$. 

\medskip

There are several possible extensions of Theorem \ref{thm-main}. 
First, we could extend the result, as in \cite{Gauckler2012}, to an energy profile $e$ with $e(l)=K+(\abs{l}-K)(1-\theta)$ instead of $e(l)=K$ for $\abs{l}\ge K$, where $0<\theta\le 1$. This shows that also the mode-energies $E_l$ for $l>K$ decay geometrically, but only with a smaller power of $\ee$ close to $1$. As in \cite{Gauckler2012}, the time-scale is then restricted to $0\le t_n\le \ee^{-\theta K(1-2\nu)/2}$.

Second, the result holds for general nonlinearities $g(u)$ instead of $u^2$ if $g$ is real-analytic near $0$ and $g(0)=g'(0)=0$. In addition, stronger estimates hold if further derivatives of $g$ vanish at $0$.

Finally, the result can be extended to more general initial energy profiles, for example to the situation where a whole band $E_l^0$, $\abs{l}\le B$, of initial mode energies is of order~$\ee$.

\subsection{Numerical experiment}\label{subsec-numexp}

We consider the nonlinear wave equation \eqref{eq-nlw} with $\rho=\sqrt{3}$ and with initial value satisfying \eqref{eq-init} for $\ee=10^{-3}$. 
We apply the trigonometric integrator \eqref{eq-method} with filter functions $\phi=1$ and $\psi=\sinc$ (the impulse method or method of Deuflhard) to the equation. We take $M=2^5$ for the spectral collocation in space, and we use three different time step-sizes for the discretization in time\footnote{The code is available at \url{http://www.waves.kit.edu/downloads/CRC1173_Preprint_2016-13_supplement.zip}}. 

\begin{figure}[t]
\, \includegraphics{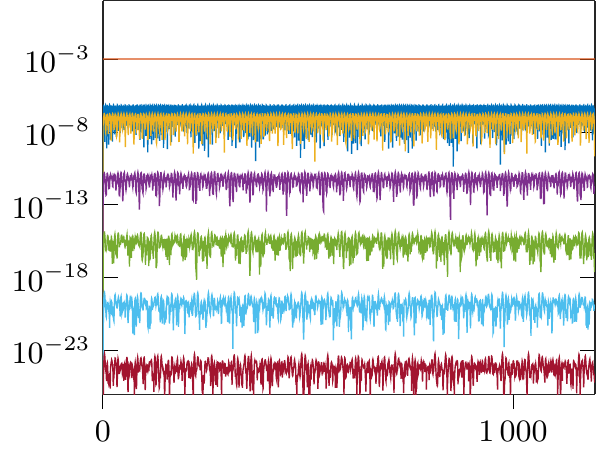} \hfill \includegraphics{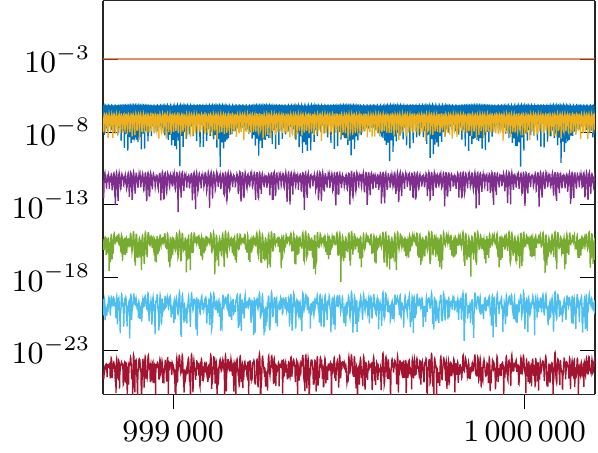} \,
\caption{Mode energies $E_l^n$ vs.\ time $t_n$ for the numerical solution with time step-size $\tau=0.05$.}\label{fig-nonres}
\end{figure}

The first time step-size is $\tau=0.05$. For this time step-size, the trigonometric integrator integrates the geometrically decaying energy strata qualitatively correctly, even on long time intervals, see Figure \ref{fig-nonres} (every tenth time step is plotted there). The time step-size fulfills the CFL-type condition \eqref{eq-cfl} with $K=6$, under which numerical resonances in the non-resonance condition \eqref{eq-nonres-gesamt} of Theorem \ref{thm-main} can be excluded.

The second time step-size is $\tau=2\pi/(\om_1+\om_6+\om_7)\approx 0.4212$. It is chosen in such a way that it does not satisfy the non-resonance condition \eqref{eq-nonres-gesamt}. The resulting numerical resonance becomes apparent in Figure \ref{fig-res}, where we observe in particular an exchange among the first, sixth and seventh mode. For better visibility, we have plotted $\max_{m=0,\dots,99}E_l^{n+m}$ for $n=0,100,200,\ldots$ instead of $E_l^n$ for $n=0,1,2,\ldots$ in Figure \ref{fig-res}. In the same figure, a similar behaviour can be observed for the resonant time step-size $\tau=2\pi/(-\om_1+\om_6+\om_7)$.

\begin{figure}
\, \includegraphics{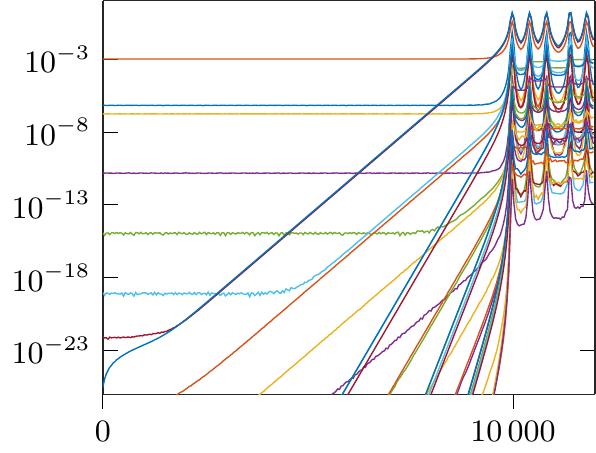} \hfill \includegraphics{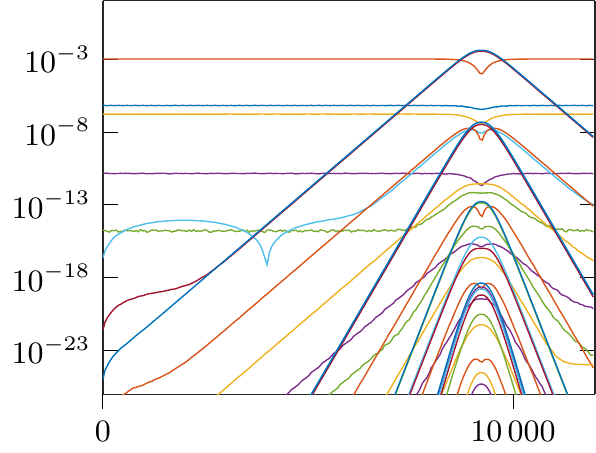} \,
\caption{Mode energies $E_l^n$ vs.\ time $t_n$ for the numerical solutions with time step-size $\tau=2\pi/(\om_1+\om_6+\om_7)$ (left) and time step-size $\tau=2\pi/(-\om_1+\om_6+\om_7)$ (right).}\label{fig-res}
\end{figure}

\section{Modulated Fourier expansions: Proof of Theorem~\ref{thm-main}}\label{sec-mfe}

We assume that the non-resonance conditions \eqref{eq-nonres-gesamt} holds for some $0\le\nu<\frac12$. 

\subsection{Approximation ansatz}

We will approximate the numerical solution by a \emph{modulated Fourier expansion},
\begin{equation}\label{eq-mfe}
u_j^n \approx \widetilde{u}_j^n = \sum_{\fk\in\calK_j} z_j^\fk \bigl(\ee^{\nu/2} t_n\bigr)\e^{\iu(\fk\cdot \fomega)t_n}, \qquad j=-M,\dots,M-1,
\end{equation}
where $\calK_j = \{ \fk : (j,\fk)\in\calK \}$ with the set $\calK$ of
\eqref{eq-calK}. We require that the \emph{modulation functions} $z_j^\fk$ are
polynomials. In contrast to the modulated Fourier expansion of \cite[Equation
  (7)]{Gauckler2012} for the exact solution, the modulation functions
considered here vary on a slow time-scale $\ee^{\nu/2}t$, where $\nu$ is the
parameter from the non-resonance condition~\eqref{eq-nonres-gesamt}. With this slow time scale, derivatives of $z_j^\fk(\ee^{\nu/2}t)$ with respect to $t$ carry additional factors $\ee^{\nu/2}$, which will be used to compensate for this factor in the weak non-resonance condition \eqref{eq-nonres-weak}.

Having in mind that $u_j^n$, $j=-M,\dots,M-1$, are the Fourier coefficients of a trigonometric polynomial of degree $M$, we write in the following $u_M^n=u_{-M}^n$ and $z_{M}^\fk=z_{-M}^\fk$.

Inserting the ansatz \eqref{eq-mfe} into the trigonometric integrator \eqref{eq-method-main} and
comparing the coefficients of $\e^{\iu(\fk\cdot\fomega)t}$, yields a set of
equations for the modulation functions $\fz = (z_j^\fk)_{(j,\fk)\in\calK}$, for which solutions are to be constructed up to a small defect $d_j^\fk$:
\begin{equation}\label{eq-modsystem}\begin{split}
\e^{\iu(\fk\cdot \fomega)\tau} z^{\fk}_j\bigl(\ee^{\nu/2}(t+\tau)\bigr) &-2\cos (\tau \omega_j) z^{\fk}_j\bigl(\ee^{\nu/2}t\bigr) + \e^{-\iu(\fk\cdot \fomega)\tau} z^{\fk}_j\bigl(\ee^{\nu/2}(t-\tau)\bigr)\\
 & = -\tau^2\psi(\tau\omega_j) \nabla_{-j}^{-\fk} \calU\bigl(\fPhi\fz(\ee^{\nu/2}t)\bigr) + \tau^2\psi(\tau\omega_j)d_j^\fk\bigl(\ee^{\nu/2}t\bigr)
\end{split}\end{equation}
with the extended potential
\begin{equation}\label{eq-calU}
\calU(\fz) = - \frac13 \sum_{j_1+j_2+j_3\equiv 0} \sum_{\fk^1+\fk^2+\fk^3=\mathbf{0}} z_{j_1}^{\fk^1} z_{j_2}^{\fk^2} z_{j_3}^{\fk^3},
\end{equation}
with $\fPhi\fz = (\phi(\tau\omega_j) z_j^\fk)_{(j,\fk)\in\calK}$ and with $\nabla_{-j}^{-\fk}$ denoting the partial derivative with respect to $z_{-j}^{-\fk}$. In \eqref{eq-calU} and in the following we denote by $\equiv$ the congruence modulo $2M$.

With the modulated Fourier expansion \eqref{eq-mfe} at hand, the formula \eqref{eq-method-velocity} for the velocity approximation leads to an approximation of $\dot{u}_j^n$ by a modulated Fourier expansion:
\begin{equation}\label{align:dot-u-mfe}\begin{split}
\dot{u}_j^n \approx \dot{\widetilde{u}}_j^n = \bigl( 2\tau \sinc(\tau\omega_j) \bigr)^{-1} \sum_{\fk\in\calK_j} \Bigl( z_j^\fk&(\ee^{\nu/2}t_{n+1}) \e^{\iu (\fk\cdot\fomega) \tau}\\ &- z_j^\fk(\ee^{\nu/2}t_{n-1}) \e^{-\iu (\fk\cdot\fomega) \tau} \Bigr)\e^{\iu (\fk\cdot\fomega) t_n}.
\end{split}\end{equation}

Finally, we get conditions from the fact that the ansatz \eqref{eq-mfe} should satisfy the initial condition:
\begin{subequations}\label{eq-modsystem-init}
\begin{align}
u_j^0 &= \sum_{\fk\in\calK_j} z_j^{\fk}(0),\label{eq-modsystem-init1}\\
\dot{u}_j^0 &= \bigl( 2\tau \sinc(\tau\omega_j) \bigr)^{-1} \sum_{\fk\in\mathcal{K}_j} 
\Bigl(
z_j^\fk(\ee^{\nu/2}\tau) \e^{\iu (\fk\cdot\fomega)\tau} - 
z_j^\fk(-\ee^{\nu/2}\tau) \e^{-\iu (\fk\cdot\fomega)\tau} 
\Bigr),\label{eq-modsystem-init2}
\end{align}
\end{subequations}
where \eqref{align:dot-u-mfe} was used for the derivation of the second equation.

\subsection{Norms}\label{subsec-norms}

For vectors $\fv=(v_{-M},\dots,v_{M-1})$ of Fourier coefficients of trigonometric polynomials of degree $M$, we consider the norm
\[
\norm{\fv} = \klabigg{\sum_{j=-M}^{M-1} \si_j \ee^{-2e(j)\nu} \abs{v_j}^2}^{1/2},
\]
where $e(j)$ is the energy profile of \eqref{eq-e} and $\si_j$ are the weights of \eqref{eq-sigma}.
For $\nu=0$, this is the Sobolev $H^s$-norm of the corresponding trigonometric polynomial $\sum_{j=-M}^{M-1} v_j \e^{\iu jx}$. 
The rescaling $\ee^{-2e(j)\nu}$ originates from the non-resonance condition \eqref{eq-nonres-weak} that introduces $\ee^{\nu/2}$.

We will make frequent use of the fact that this norm behaves well with respect to the discrete convolution,
\begin{equation}\label{eq-algebra}
\norm{\fu \ast \fv } \le C \norm{\fu} \, \norm{\fv}.
\end{equation}
This follows from the corresponding property of the Sobolev $H^s$-norm,
see, for example, \cite[Lemma 4.2]{Hairer2008}, and the fact that $e(j)\le e(j_1)+e(j_2)$ for $j\equiv j_1+j_2$. Similarly, we also have
\begin{equation}\label{eq-algebra2}
\norm{\fOmega(\fu \ast \fv) } \le C \norm{\fOmega\fu} \, \norm{\fOmega\fv}.
\end{equation}
In particular, we have for $\nu=0$ in \eqref{eq-algebra} 
\begin{equation}\label{eq-algebra-basic}
\sum_{j=-M}^{M-1} \si_j \absbigg{ \sum_{j_1+j_2\equiv j} u_{j_1}v_{j_2} }^2 \le C^2 \biggl( \sum_{j_1=-M}^{M-1} \si_{j_1} \abs{u_{j_1}}^2 \biggr) \biggl( \sum_{j_2=-M}^{M-1} \si_{j_2} \abs{v_{j_2}}^2 \biggr) .
\end{equation}

\subsection{The modulated Fourier expansion on a short time interval}

Instead of initial mode energies \eqref{eq-init}, we consider the more general choice 
\begin{equation}\label{eq-modeenergies-init}
\sum_{l=0}^M \sigma_l \ee^{-e(l)} E_{l}^{0} \le C_0
\end{equation}
of initial mode energies, 
which is the expected situation at later times (see Theorem~\ref{thm-main}). 
We then have the following discrete counterpart of \cite[Theorem 3]{Gauckler2012}, whose proof will be given in Section~\ref{sec-proof1}. 

\begin{theorem}\label{thm-3}
Under the assumptions of Theorem \ref{thm-main}, but with \eqref{eq-modeenergies-init} instead of \eqref{eq-init}, the numerical solution $\fu^n=(u_{-M}^n,\dots,u_{M-1}^n)^\transpose$ admits an expansion
\[
u_j^n = \sum_{\fk\in\calK_j} z_j^\fk\klabig{\ee^{\nu/2}t_n} \e^{\iu (\fk\cdot\fomega) t_n} + r_j^n \myfor 0\le t_n=n\tau\le 1,
\]
where $\calK_j=\{\,\fk : (j,\fk)\in\calK \,\}$ with the set $\calK$ of \eqref{eq-calK}, where the coefficient functions $z_j^\fk$ are polynomials satisfying $z_{-j}^{-\fk}=\overline{z_j^\fk}$, and where the remainder term $\fr^n=(r_{-M}^n,\dots,r_{M-1}^n)^\transpose$ and the corresponding remainder term $\dot{\fr}^n=(\dot{r}_{-M}^n,\dots,\dot{r}_{M-1}^n)^\transpose$ in the velocity approximation \eqref{align:dot-u-mfe} are bounded by
\[
\norm{\fOmega \fr^n} + \norm{\dot{\fr}^n} \le C \ee^{K(1-2\nu)} \myfor 0\le t_n=n\tau\le 1
\]
and the defect $d_j(t) = \sum_{\fk\in\calK_j} \abs{d_j^\fk(\ee^{\nu/2}t)}$ in \eqref{eq-modsystem} is bounded by
\[
\norm{\fd(t)} \le C\ee^{K(1-2\nu)} \myfor 0\le t\le 1. 
\]
The constant $C$ is independent of $\ee$, $\tau$ and $M$, but depends on $K$, on $\nu$ and $\gamma$ of \eqref{eq-nonres-gesamt}, on $\rho$ of \eqref{eq-nlw}, on $s$ of \eqref{eq-sigma} and on $C_0$ of \eqref{eq-modeenergies-init}.
\end{theorem}

\subsection{Almost-invariant energies}

We now derive almost-invariant energies of the modulation system \eqref{eq-modsystem} which enable us to consider long time intervals. The derivation of these almost-invariant energies is similar as in the case of the exact solution, see \cite[Section 3.5]{Gauckler2012}, but it now leads to discrete almost-invariant energies that involve additionally the time step-size $\tau$ and are related to those of \cite{Cohen2008a}.

Let 
\begin{equation}\label{eq-calE}
 \calE_l(t) = - \frac{\iu}{2}\sum_{j=-M}^{M-1}\sum_{\fk\in\calK_j} \frac{k_l\omega_l}{\tau\sinc (\tau\omega_j)} \, z_{-j}^{-\fk}\bigl(\ee^{\nu/2}t\bigr) \e^{\iu(\fk\cdot \fomega)\tau} z^{\fk}_j\bigl(\ee^{\nu/2}(t+\tau)\bigr).
\end{equation}
Our aim is to prove that
\begin{equation}\label{align:FastInvarianz}
\calE_l(t) = \calE_l(t-\tau) - \frac{\iu}2 \sum_{j=-M}^{M-1} \sum_{\fk\in\calK_j} \tau k_l\om_l \phi(\tau\omega_j) z_{-j}^{-\fk}\bigl(\ee^{\nu/2}t\bigr) d_j^\fk\bigl(\ee^{\nu/2}t\bigr),
\end{equation}
which shows that, for small defect $d_j^\fk$, the quantity $\calE_l$ is almost invariant. To see this almost-invariance, we use, as in the case of the exact solution, that the extended potential $\calU$ of \eqref{eq-calU} is invariant under the transformation $\calS_\flambda(\theta)\fz:=(\e^{\iu (\fk\cdot\flambda)\theta}z_j^\fk)_{(j,\fk)\in\calK}$ for $\theta\in\R$ and real vectors $\flambda=(\lambda_0,\dots,\lambda_M)$. This invariance shows that
\[
0 = \frac{\drm}{\drm \theta}\biggr|_{\theta=0} \calU\bigl(\calS_\flambda(\theta)\fPhi\fz\bigr)  = -\sum_{j=-M}^{M-1}\sum_{\fk\in\calK_j} \iu(\fk\cdot\flambda) \phi(\tau\omega_j) z_{-j}^{-\fk} \nabla_{-j}^{-\fk}\calU(\fPhi\fz).
\]
We then multiply this equation with $\tau/2$, we replace $-\nabla_{-j}^{-\fk}\calU$ with the help of \eqref{eq-modsystem}, and we use the symplecticity \eqref{eq-symplectic} of the method. This yields
\begin{align*}
0 &= \frac{\iu}{2} \sum_{j=-M}^{M-1}\sum_{\fk\in\calK_j} \frac{\fk\cdot\flambda}{\tau\sinc(\tau\omega_j)} \, z_{-j}^{-\fk}\bigl(\ee^{\nu/2}t\bigr) \klaBig{ \e^{\iu(\fk\cdot \fomega)\tau} z^{\fk}_j\bigl(\ee^{\nu/2}(t+\tau)\bigr)\\ & \qquad -2\cos (\tau \omega_j) z^{\fk}_j\bigl(\ee^{\nu/2}t\bigr)
 + \e^{-\iu(\fk\cdot \fomega)\tau} z^{\fk}_j\bigl(\ee^{\nu/2}(t-\tau)\bigr) - \tau^2\psi(\tau\omega_j)d_j^\fk\bigl(\ee^{\nu/2}t\bigr) }. 
\end{align*}
In this relation, we use that
\begin{equation}\label{eq-derivation-calE}
\sum_{j=-M}^{M-1} \sum_{\fk\in\calK_j} \frac{\fk\cdot\flambda}{\sinc (\tau\omega_j)} ß, z_{-j}^{-\fk}\bigl(\ee^{\nu/2}t\bigr) \cos (\tau \omega_j) z^{\fk}_j\bigl(\ee^{\nu/2}t\bigr) = 0,
\end{equation}
which follows from the symmetry in $j$ (recall that $z_{M}^\fk=z_{-M}^\fk$) and the asymmetry in $\fk$. Choosing $\flambda=\lvec$ with $l=0,\dots,M$ and using the definition \eqref{eq-calE} of $\calE_l$, we then end up with \eqref{align:FastInvarianz}. 

Note that, due to the appearance of the factor $k_l$ in $\calE_l$ of \eqref{eq-calE}, only modulation functions $z_j^\fk$ with $k_l\ne 0$ are actually relevant in $\calE_l$ (recall that $k_l$ is the $l$th entry of the vector $\fk$). This will be used all the time in Section~\ref{sec-proof2} below. In particular, we will use there that only very specific vectors $\fk$ with $k_l\ne 0$ for $l\ge K$ appear in the set $\calK$ of \eqref{eq-calK}.

Using a Taylor expansion of $z_j^\fk(\ee^{\nu/2}(t+\tau))$ and the property \eqref{eq-derivation-calE}, we can derive the following alternative form of $\calE_l$:
\begin{align}
\calE_l(t) &= \frac{1}{2} \sum_{j=-M}^{M-1} \sum_{\fk\in\calK_j} \biggl( k_l\omega_l (\fk\cdot\fomega) \frac{\sinc(\tau(\fk\cdot\fomega))}{\sinc (\tau\omega_j)} \, \absbig{z_j^\fk\bigl(\ee^{\nu/2}t\bigr)}^2\label{eq-calE-alt}\\
 &\qquad - \iu \frac{k_l \omega_l}{\sinc (\tau\omega_j)} \, z_{-j}^{-\fk}\bigl(\ee^{\nu/2}t\bigr) \e^{\iu(\fk\cdot \fomega)\tau} \Bigl( \ee^{\nu/2} \dot{z}^{\fk}_j\bigl(\ee^{\nu/2}t\bigr) + \tfrac12 \tau \ee^{\nu} \ddot{z}^{\fk}_j\bigl(\ee^{\nu/2}t\bigr) + \ldots \Bigr) \biggr).\notag
\end{align}
Using \eqref{eq-derivation-calE}, we see that this quantity coincides, for
$\nu=0$ and in the limits $\tau\to 0$ and $M\to\infty$, with the almost-invariant energy
of \cite[Equation (23)]{Gauckler2012} for the exact solution. We therefore
also call $\calE_l(t)$ an \emph{almost-invariant energy}. In Sections \ref{sec-proof1} and \ref{sec-proof2} below, we prove the following discrete counterparts of \cite[Theorems 4--7]{Gauckler2012} for this new almost-invariant energy. 

\begin{theorem}[Almost-invariant energies controlled by mode energies]\label{thm-4}
Under the conditions of Theorem~\ref{thm-3} we have, for $0\le t\le 1$,
\begin{equation}\label{eq-almostinvariantenergies-bound}
\sum_{l=0}^M \sigma_l \ee^{-e(l)} \abs{\calE_l(t)} \le \calC_0,
\end{equation}
where $\calC_0$ is independent of $\ee$, $\tau$ and $M$ and depends on the initial data only through the constant $C_0$ of \eqref{eq-modeenergies-init}.
\end{theorem}

\begin{theorem}[Variation of almost-invariant energies]\label{thm-5}
Under the conditions of Theorem~\ref{thm-3} we have, for $0\le t_n\le 1$,
\[
\sum_{l=0}^M \sigma_l \ee^{-e(l)} \absbig{\calE_l(t_n)-\calE_l(0)}  \le C\ee^{K(1-2\nu)/2}
\]
and $\ee^{-e(1)} \abs{\calE_1(t_n)-\calE_1(0)}  \le C\ee^{(K-1)(1-2\nu)/2} \ee^{K(1-2\nu)/2}$, 
where $C$ is independent of $\ee$, $\tau$ and $M$ and depends on the initial data only through the constant $C_0$ of \eqref{eq-modeenergies-init}.
\end{theorem}

At time $t=1$, for which we assume without loss of generality that $t_N=N\tau=1$ for some $N$, we consider a new modulated Fourier expansion leading to new almost-invariant energies. 

\begin{theorem}[Transitions in the almost-invariant energies]\label{thm-6}
Let the conditions of Theorem~\ref{thm-3} be fulfilled. Let $\fz(\ee^{\nu/2} t)=(z^{\fk}_j(\ee^{\nu/2} t))_{(j,\fk)\in\calK}$ for $0\le t \le 1$ be the coefficient functions as in Theorem \ref{thm-3} with corresponding almost-invariant energies $\calE_l(t)$ for initial data $(\fu^0,\dot{\fu}^0)$. Let further $\widetilde\fz(\ee^{\nu/2} t)=({\widetilde z}^{\fk}_j(\ee^{\nu/2} t))_{(j,\fk)\in\calK}$ be the coefficient functions and $\widetilde{\calE}(t)$ the corresponding almost-invariants of the modulated Fourier expansion for $0\le t \le 1$ corresponding to initial data $(\fu^N,\dot{\fu}^N)$ with $t_N=N\tau=1$, constructed as in Theorem \ref{thm-3}.
If $(\fu^N,\dot{\fu}^N)$ also satisfies the bound \eqref{eq-modeenergies-init}, then
\[
\sum_{l=0}^M \sigma_l \ee^{-e(l)} \absbig{ \calE_l(1) - \widetilde{\calE}_l(0) } \le C \ee^{K(1-2\nu)/2}
\]
and 
$\ee^{-e(1)} \abs{\calE_1(1) - \widetilde{\calE}_1(0)}  \le C\ee^{(K-1)(1-2\nu)/2} \ee^{K(1-2\nu)/2}$, 
where $C$ is independent of $\ee$, $\tau$ and $M$ and depends on the initial data only through the constant $C_0$ of \eqref{eq-modeenergies-init}.
\end{theorem}

\begin{theorem}[Mode energies controlled by almost-invariant energies]\label{thm-7}
Let the conditions of Theorem \ref{thm-3} be fulfilled. If the almost-invariant energies satisfy \eqref{eq-almostinvariantenergies-bound} for $0\le t_n=n\tau \le 1$,
then the mode energies are bounded by
\[
\sum_{l=0}^M \sigma_l \ee^{-e(l)} E_{l}^n \le \calC
\]
and $\ee^{-e(1)} \abs{E_1^n - \calE_1(t_n)}\le \calC \ee^{1-2\nu}$, 
where $\calC$ depends on $\calC_0$ in \eqref{eq-almostinvariantenergies-bound},
but is independent of $\ee$, $\tau$, $M$ and $C_0$ of
\eqref{eq-modeenergies-init} if $\ee^{1-2\nu}$ is sufficiently small.
\end{theorem}

\subsection{From short to long time intervals: Proof of Theorem \ref{thm-main}}

Based on Theorems \ref{thm-3}--\ref{thm-7}, the proof of Theorem \ref{thm-main} is the same as in the case of the exact solution, see \cite[Section 3.6]{Gauckler2012}: Theorems \ref{thm-3}--\ref{thm-5} and \ref{thm-7} yield the statement of Theorem \ref{thm-main} on a short time interval $0\le t_n\le 1$. Theorem \ref{thm-6} can be used to patch many of these short time intervals together, on which the almost-invariant energies $\calE_l$ are still well preserved (Theorems~\ref{thm-5} and \ref{thm-6}) and allow to control the mode energies $E_l$ (Theorem~\ref{thm-7}). A schematic overview of the proof is given in Figure \ref{fig-proof-thm-main}.

\begin{figure}
\centering
\includegraphics{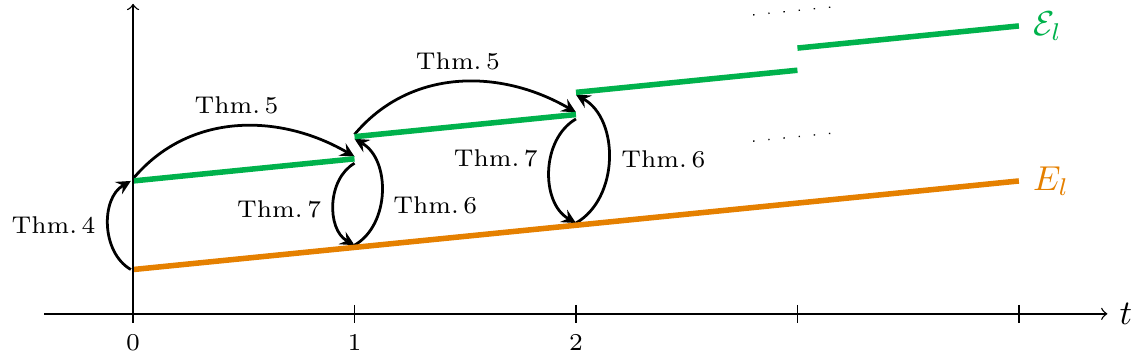}
\caption{A schematic overview of the proof of Theorem \ref{thm-main}: How the results on the almost-invariant energies $\calE_l$ are applied to control the mode energies $E_l$ and vice versa.}\label{fig-proof-thm-main}
\end{figure}

\section{Construction of a modulated Fourier expansion: Proofs of Theorems \ref{thm-3} and \ref{thm-4}}\label{sec-proof1}

Throughout this section, we work under the assumptions of Theorem \ref{thm-3}. In addition, we let 
\[
\de = \ee^{1/2},
\]
such that $u_{j}^0,\dot{u}_{j}^0=\mathcal{O}(\de^{e(l)})$.
Moreover, we write $\lesssim$ for an inequality up to a factor that is independent of $\de = \ee^{1/2}$, $\tau$ and $M$.

\subsection{Expansion of the modulation functions}

The construction of modulation functions $z_j^\fk$ of the modulated Fourier expansion \eqref{eq-mfe} is based on an expansion
\begin{equation}\label{eq-modfunc-exp}\begin{split}
z_{j}^{\fk} &= \delta^{m(j,\fk)} \sum_{m= m(j,\fk)}^{2K-1} \delta^{(m-m(j,\fk))(1-2\nu)} z^{\fk}_{j,m}\\
 &= \delta^{2m(j,\fk)\nu} \sum_{m= m(j,\fk)}^{2K-1} \delta^{m(1-2\nu)} z^{\fk}_{j,m}
\end{split}\end{equation}
with polynomials $z_{j,m}^\fk=z_{j,m}^\fk(\de^\nu t)$ and with
\[
m(j,\fk) = \max\Bigl(e(j),\max_{l \colon k_l\ne 0} e(l) \Bigr).
\]
This can be motivated by the fact that on the one hand we expect
$z_j^\fk=\calO(\de^{m(j,\fk)})$ from the analysis of the exact solution in
\cite{Gauckler2012}, but on the other hand we do not expect an expansion in true powers of $\de$ as in \cite{Gauckler2012} because of the non-resonance condition \eqref{eq-nonres-gesamt} involving $\de^{\nu}$.

Our goal now is to derive equations for the modulation coefficient functions $z_{j,m}^\fk$ in the ansatz \eqref{eq-modfunc-exp}.

Inserting this ansatz into \eqref{eq-modsystem}, expanding $z_j^\fk(\de^\nu(t\pm\tau))$ in a Taylor series and requiring that powers of $\de^{1-2\nu}$ agree on both sides
yields, neglecting the defect $d_j^k$,
\begin{equation}\label{eq-modsystem-coeff}\begin{split}
4s_{\langle j\rangle-\fk}s_{\langle j\rangle+\fk}  &z^{\fk}_{j,m} + 2\iu\tau \delta^{\nu} s_{2\fk} \dot{z}^{\fk}_{j,m} + \tau^2 \delta^{2\nu} c_{2\fk} \ddot{z}^{\fk}_{j,m} +\hdots 	\\
& = \tau^2\psi(\tau\omega_j) \sum_{\fk^1+\fk^2=\fk} \sum_{j_1+j_2\equiv j} \delta^{2(m(j_1,\fk^1)+m(j_2,\fk^2)-m(j,\fk))\nu}\\
&\qquad\qquad\qquad\qquad\qquad \sum_{m_1+m_2=m} \phi(\tau\omega_{j_1}) z^{\fk^1}_{j_1,m_1} \phi(\tau\omega_{j_2}) z^{\fk^2}_{j_2,m_2}.
\end{split}\end{equation}
Here, we use the notation $s_\fk=\sin(\frac{\tau}{2} \fk\cdot\fomega)$ and $c_\fk=\cos(\frac{\tau}{2} \fk\cdot\fomega)$ and the fact that $\cos x- \cos y = 2\sin(\frac{y-x}{2})\sin(\frac{y+x}{2})$. 
All functions in \eqref{eq-modsystem-coeff} are evaluated at $\de^\nu t$ and dots on the functions $z_{j,m}^\fk$ denote
derivatives with respect to the slow time scale $\de^\nu t$.
Note that $m(j,\fk) \le m(j_1,\fk^1)+m(j_2,\fk^2)$ if $\fk=\fk^1+\fk^2$ and $j\equiv j_1+j_2\bmod{2M}$, and hence the power of $\de$ on the right-hand side of \eqref{eq-modsystem-coeff} is small. We recall that $\equiv$ denotes the congruence modulo $2M$.

\begin{remark}
For $\nu=0$ we recover in \eqref{eq-modsystem-coeff}, after division by $\tau^2$ and in the limit $\tau\to 0$, the system of equations which was used in the case of the exact solution, see \cite[Equation (28)]{Gauckler2012}.
For $\nu>0$, the above construction becomes significantly more involved than there. The reason is that the non-resonance condition \eqref{eq-nonres-weak} only allows us to bound the factor $s_{\langle j\rangle-\fk}s_{\langle j\rangle+\fk}$ on the left-hand side of \eqref{eq-modsystem-coeff} from below by $\gamma^2 \tau^2 \de^{2\nu}$. As we will see in the proof of Lemma \ref{lemma-bounds} below, we can compensate this possibly small factor with an additional $\de^{2\nu}$ on the right-hand side (together with $\tau^2$), which we gain from the special choice \eqref{eq-modfunc-exp} as ansatz for $z_j^\fk$. 
\end{remark}

In addition to \eqref{eq-modsystem-coeff}, we get from condition \eqref{eq-modsystem-init} that
\begin{subequations}\label{eq-modsystem-coeff-init}
\begin{align}
&\sum_{\fk\in\mathcal{K}_j} \de^{2(m(j,\fk)-e(j))\nu} z_{j,m}^\fk(0) = \begin{cases}
\de^{-e(j)} u_{j}^0, & m = e(j),\\
0, & \text{else,}
\end{cases}\\
&\sum_{\fk\in\mathcal{K}_j} \de^{2(m(j,\fk)-e(j))\nu} \Bigl( 2\iu s_{2\fk} z_{j,m}^\fk(0) + 2\tau\delta^{\nu} c_{2\fk} \dot{z}_{j,m}^\fk(0) + \iu \tau^2\delta^{2\nu} s_{2\fk} \ddot{z}_{j,m}^\fk(0) + \ldots \Bigr)\notag\\
&\qquad\qquad\qquad\qquad\qquad\qquad = \begin{cases}
2\tau \sinc(\tau \omega_j) \de^{-e(j)} \dot{u}_{j}^0, & m = e(j),\\
0, & \text{else,}
\end{cases}
\end{align} 
\end{subequations}
where we use again a Taylor expansion of the modulation functions.

\subsection{Construction of modulation functions}

We construct polynomial modulation functions of the form \eqref{eq-modfunc-exp} by solving \eqref{eq-modsystem-coeff} and \eqref{eq-modsystem-coeff-init} consecutively for $m=1,2,\dots,2K-1$. For convenience, we set $z_{j,m'}^\fk=0$ for $m'<m(j,\fk)$. Assuming that we have computed all polynomials $z_{j,m'}^\fk$ for $m'<m$, the construction relies on the observation that only already computed functions $z_{j',m'}^{\fk'}$ appear on the right-hand side of \eqref{eq-modsystem-coeff}. This equation is thus, for $z(s)=z_{j,m}^\fk(s)$, of the form
\begin{equation}\label{eq-construction}
\alpha_0 z(s) - \alpha_1 \dot{z}(s) - \alpha_2 \ddot{z}(s) - \dots - \alpha_L z^{(L)}(s) = p(s)
\end{equation}
with coefficients $\al_0,\dots,\al_L$ and a polynomial $p$.

For $\fk\in\calK_j$ with $\fk\ne\pm\jvec$, the coefficient $\al_0$ in this equation is nonzero (by the non-resonance condition \eqref{eq-nonres-gesamt}), and the unique polynomial solution of this equation is given by
\begin{equation}\label{eq-construction-offdiag}
z(s) = \sum_{k=0}^{\deg(p)} \biggl( \frac{\alpha_1}{\al_0} \frac{\drm}{\drm s} + \frac{\alpha_2}{\al_0} \frac{\drm^2}{\drm s^2} + \dots + \frac{\alpha_L}{\al_0} \frac{\drm^L}{\drm s^L} \biggr)^k \frac{1}{\al_0} \, p(s).
\end{equation}
The modulation coefficient functions constructed in this way are called off-diagonal modulation coefficient functions.

For $\fk=\pm\jvec$, the coefficient $\al_0$ is zero, and the polynomial solutions of~\eqref{eq-construction} are given by
\begin{subequations}\label{eq-construction-diag}
\begin{equation}\label{eq-construction-diag-a}
z(s) = z(0) + \int_0^s \dot{z}(\si) \,\drm \si
\end{equation}
with
\begin{equation}\label{eq-construction-diag-der}
\dot{z}(s) = \sum_{k=0}^{\deg(p)} \biggl( \frac{\alpha_2}{\al_1} \frac{\drm}{\drm s} + \dots + \frac{\alpha_L}{\al_1} \frac{\drm^{L-1}}{\drm s^{L-1}} \biggr)^k \frac{(-1)^{k+1}}{\al_1} \, p(s).
\end{equation}
\end{subequations}
In \eqref{eq-construction-diag-a}, the initial value $z(0)$ is still a free parameter. We use
\eqref{eq-modsystem-coeff-init} to fix it. Adding and subtracting the two equations of \eqref{eq-modsystem-coeff-init}, after multiplying the first equation with $2\iu\om_j$ and the second one with $\om_j/s_{2\jvec}=1/(\tau\sinc(\tau\om_j))$, gives us
\begin{align}
2\iu \om_j z_{j,m}^{\pm\jvec}(0) &= - \iu \!\! \sum_{\substack{\fk\in\calK_j\\ \fk\ne\jvec, \, \fk\ne-\jvec}} \!\! \de^{2(m(j,\fk)-e(j))\nu} \biggl( \om_j \pm (\fk\cdot\fomega) \frac{\sinc(\tau(\fk\cdot\fomega))}{\sinc(\tau\om_j)} \biggr) z_{j,m}^\fk(0)\notag\\
 &\qquad \mp \sum_{\fk\in\calK_j} \frac{\de^{2(m(j,\fk)-e(j))\nu}}{\sinc(\tau\om_j)} \Bigl( \de^\nu c_{2\fk} \dot{z}_{j,m}^\fk(0) + \tfrac{\iu}2 \tau\de^{2\nu} s_{2\fk} \ddot{z}_{j,m}^\fk(0) + \dots \Bigr)\notag\\
 &\qquad + \de^{-e(j)} \begin{cases} \iu \om_ju_{j}^0 \pm \dot{u}_{j}^0, & m = e(j),\\ 0, & \text{else.} \end{cases} \label{eq-construction-init}
\end{align}
Note again that this equation becomes for $\nu=0$ and in the limit $\tau\to 0$ the corresponding equation for the exact solution, see \cite[Equation (31)]{Gauckler2012}. The modulation coefficient functions constructed with \eqref{eq-construction-diag} and \eqref{eq-construction-init} are called diagonal modulation coefficient functions.


\begin{example}\label{ex-constr} 
We illustrate the above construction for the first two nontrivial values of $m$. For simplicity, we restrict to the case $\nu=0$ and $\phi=1$ (and hence $\psi=\sinc$ by \eqref{eq-symplectic}).

For $m=1$, we observe that the polynomial $p$ in \eqref{eq-construction}, which is the right-hand side of \eqref{eq-modsystem-coeff}, vanishes for all $j$ and all $\fk$. The above construction \eqref{eq-construction-offdiag} thus yields
\[
z_{j,1}^\fk \equiv 0 \myfor \fk\ne \pm\jvec,
\]
and \eqref{eq-construction-diag-der} yields $\dot{z}_{j,1}^{\pm\jvec}\equiv 0$. The computation of the initial values for $z_{j,1}^{\pm\jvec}$ with \eqref{eq-construction-init} finally yields 
\[
z_{j,1}^{\pm\jvec} \equiv 0 \myfor \abs{j}\ne 1
\]
and
\[
z_{1,1}^{\pm\skla{1}} \equiv \tfrac1{2\iu\om_1} \de^{-1} \bigl(\iu\om_1 u_{1}^0 \pm \dot{u}_{1}^0\bigr), \qquad z_{-1,1}^{\pm\skla{1}} \equiv \tfrac1{2\iu\om_1} \de^{-1} \bigl(\iu\om_1 u_{-1}^0 \pm \dot{u}_{-1}^0\bigr)
\]
by \eqref{eq-construction-diag-a}.

For $m=2$, we observe that the polynomial $p$ in \eqref{eq-construction}, which is the right-hand side of \eqref{eq-modsystem-coeff}, is non-zero only for $j\in\{\pm 2,0\}$ and $\fk=\{\pm2\skla{1},\mathbf{0}\}$. In particular, we get from \eqref{eq-construction-offdiag}
\[
z_{j,2}^\fk \equiv 0 \myfor \fk\ne \pm\jvec, \quad (j,\fk)\notin \{\pm 2,0\} \times \{\pm2\skla{1},\mathbf{0}\}.
\]
For $j=2$ and $\fk=2\skla{1}$, the polynomial $p$ is non-zero but constant, and we get from \eqref{eq-construction-offdiag} the non-zero off-diagonal modulation coefficient function
\begin{align*}
z_{2,2}^{2\skla{1}} &= \frac{\tau^2\sinc(\tau \om_2)}{4 s_{\skla{2}-2\skla{1}} s_{\skla{2}+2\skla{1}}} \, z_{1,1}^{\skla{1}} \cdot z_{1,1}^{\skla{1}},
\end{align*}
and similar expressions for $z_{2,2}^{\mathbf{0}}$, $z_{2,2}^{-2\skla{1}}$, $z_{-2,2}^{\pm2\skla{1}}$, $z_{-2,2}^{\mathbf{0}}$, $z_{0,2}^{\pm2\skla{1}}$ and $z_{0,2}^{\mathbf{0}}$. For the diagonal modulation coefficient functions, we get $\dot{z}_{j,2}^{\pm\jvec}\equiv 0$ from \eqref{eq-construction-diag-der}, and the computation of the initial values with \eqref{eq-construction-init} finally yields by \eqref{eq-construction-diag-a}
\[
z_{j,2}^{\pm\jvec} \equiv 0 \myfor j\notin \{\pm 2,0\}
\]
and 
\begin{align*}
z_{2,2}^{\pm\skla{2}} \equiv \frac1{2\iu\om_2} \biggl( &-\iu (\om_2\pm2\om_1) \, \frac{\sinc(2\tau\om_1)}{\sinc(\tau\om_2)}\, z_{2,2}^{2\skla{1}} -\iu \om_2 z_{2,2}^{\mathbf{0}}\\
 &- \iu (\om_2\mp2\om_1) \, \frac{\sinc(-2\tau\om_1)}{\sinc(\tau\om_2)}\, z_{2,2}^{-2\skla{1}} + \de^{-2} \bigl(\iu\om_2 u_2^0 \pm \dot{u}_2^0\bigr) \biggr)
\end{align*}
and similar expressions for $z_{-2,2}^{\pm\skla{-2}}$ and $z_{0,2}^{\pm\skla{0}}$.

\begin{figure}[t]
\, \includegraphics{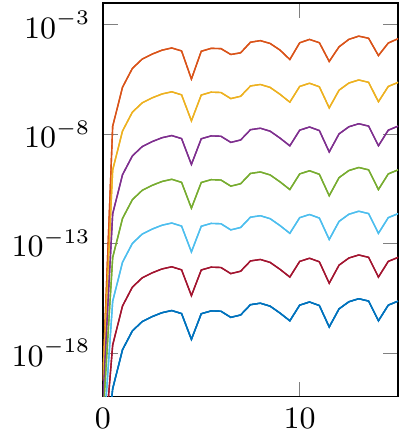} \, \includegraphics{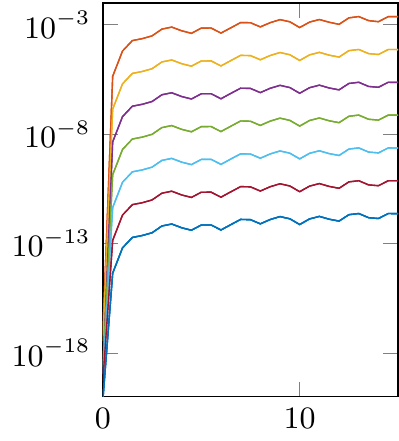} \, \includegraphics{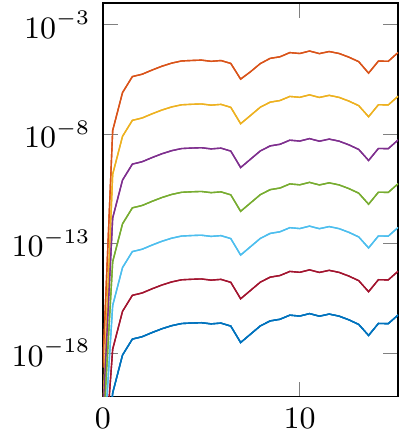} \,
\caption{Approximation error $\abs{u_j^n-\widehat{u}_j^n}$ vs.\ time $t_n$ for $j=0$ (left), $j=1$ (middle) and $j=2$ (right). Different lines correspond to different values of $\ee=10^{-2},10^{-3},\ldots,10^{-8}$ in \eqref{eq-init}.}\label{fig-modfuncs}
\end{figure}

By means of the expansions \eqref{eq-mfe} and \eqref{eq-modfunc-exp} (the latter truncated at $m=2$ instead of $m=2K-1$), we get from the computed modulation coefficient functions $z_{j,m}^\fk$ with $m=1,2$ approximations (recall that $\nu=0$ here)
\[
\widehat{u}_j^n = \sum_{\fk\in\calK_j} \sum_{m= m(j,\fk)}^{2} \delta^{m} z^{\fk}_{j,m}(t_n) \e^{\iu(\fk\cdot\fomega) t_n}
\]
to the numerical solution $u_j^n$. For $\tau=0.05$ and $\rho=\sqrt{3}$, the quality of this approximation is illustrated in Figure \ref{fig-modfuncs} for initial values satisfying \eqref{eq-init} with different values of $\ee$. We observe that the approximation is exact for $n=0$ (by construction) and that the approximation is of order $\de^3$ for $j=1$ on a time interval of order one. We also observe an improved approximation order of $\de^4$ for $j=0$ and $j=2$. This can be explained with the construction for $m=3$, which yields $z_{2,3}^\fk=z_{0,3}^\fk\equiv 0$ for all $\fk$.
\end{example}

Going beyond this example, we verify as in \cite[Section 4.1]{Gauckler2012} the following properties of the constructed functions $z_{j,m}^\fk$. Using $e(j)\le e(j_1)+e(j_2)$ for $j\equiv j_1+j_2$, $\mu(\fk)\le \mu(\fk^1)+\mu(\fk^2)$ for $\fk=\fk^1+\fk^2$ and $e(j)=\mu(\fk)<\mu(\fk^1)+\mu(\fk^2)$ for $\fk=\pm\jvec=\fk^1+\fk^2$ with $\mu$ defined in \eqref{eq-mu}, we see that they are polynomials of degree
\begin{equation}\label{eq-degree}
\deg\klabig{z_{j,m}^\fk} \le m-\max(e(j),\mu(\fk)),
\end{equation}
where a negative degree corresponds to the zero polynomial. Moreover, the polynomial $z_{j,m}^\fk$ can be different from the zero polynomial only in the two cases
\begin{align}
\text{case 1:} \quad & \abs{j}\le m, \quad \mu(\fk)\le m, \quad k_l= 0 \,\text{ for $l\ge \min(m+1,K)$,} \label{eq-calK-case1}\\
\text{case 2:} \quad & m\ge K, \quad \abs{(j-r) \bmod{2M}}\ge K, \quad \abs{r}\le m-K, \notag\\
 &\qquad\qquad\,\,\; \fk =\pm\skla{(j-r) \bmod{2M}} + \bar{\fk}, \quad \mu(\bar{\fk})\le m-K. \label{eq-calK-case2}
\end{align}
This explains how the set $\calK$ of \eqref{eq-calK} is built up. 
The two cases follow from the decomposition
\[
2 \phi(\tau\omega_{j_1}) \phi(\tau\omega_{j_2}) \sum_{l=1}^{m-K}  z^{\fk^1}_{j_1,m-l} z^{\fk^2}_{j_2,l} + \phi(\tau\omega_{j_1}) \phi(\tau\omega_{j_2}) \sum_{\substack{m_1+m_2=m\\ m_1<K, \, m_2<K}} z^{\fk^1}_{j_1,m_1} z^{\fk^2}_{j_2,m_2}
\]
of the last line in \eqref{eq-modsystem-coeff}, where the first term is only present for $m>K$: the functions $z^{\fk^i}_{j_i,m_i}$ in the second sum and the function $z^{\fk^2}_{j_2,l}$ in the first sum belong inductively to the first case above, whereas the function $z^{\fk^1}_{j_1,m-l}$ in the first sum belongs either to the first or to the second case, leading to a function $z_{j,m}^\fk$ belonging to the first or second case, respectively.
In addition, we have $\overline{z_{j,m}^\fk}=z_{-j,m}^{-\fk}$ and 
\begin{equation}\label{eq-specialzero}
z_{j,e(l)}^{\pm\skla{l}}=0 \myif \abs{l}\ne\abs{j}.
\end{equation}

\subsection{Bounds of the modulation functions}

For polynomials $z=z(s)$, we introduce the norm
\[
\normv{z}_t = \sum_{l\ge 0} \frac{1}{l!} \absbigg{\frac{\drm^l}{\drm s^l} z(s)}_{s=\de^\nu t}
\]
for $0\le t\le 1$. It has the properties that $\normv{z\cdot w}_t \le \normv{z}_t\cdot\normv{w}_t$ and $\normv{\dot{v}}_t \le \deg(v) \, \normv{v}_t$. With this norm, we have the following discrete counterpart of \cite[Lemma~1]{Gauckler2012}. In contrast to the analytical result in \cite{Gauckler2012}, we gain a factor $\abs{\sinc(\tau\om_j)}^{-1}$ in the estimates, which comes from the filter $\psi$ in the trigonometric integrator \eqref{eq-method} and will be needed later on to deal with this factor in the almost-invariant energies $\calE_l$ of \eqref{eq-calE}.

\begin{lemma}\label{lemma-bounds}
For $m=1,\dots,2K-1$ and $0\le t\le 1$ we have
\begin{gather*}
\sum_{j=-M}^{M-1} \si_j \biggl( \abs{\sinc(\tau\om_j)}^{-1} \sum_{\pm\jvec\ne \fk\in\calK_j} \ga_j^\fk \normvbig{z_{j,m}^\fk}_t \biggr)^2 \lesssim 1,\\
\sum_{j=-M}^{M-1} \si_j \Bigl( \abs{\sinc(\tau\om_j)}^{-1} \ga_j^{\pm\jvec} \normvbig{\dot{z}_{j,m}^{\pm\jvec}}_t \Bigr)^2 \lesssim 1, \qquad
\sum_{j=-M}^{M-1} \si_j \Bigl( \ga_j^{\pm\jvec} \normvbig{z_{j,m}^{\pm\jvec}}_t \Bigr)^2 \lesssim 1
\end{gather*}
with the additional weight
\[
\ga_j^\fk = \max(1,\om_j,\abs{\fk\cdot\fomega}).
\]
\end{lemma}
\begin{proof}
The statement is shown by induction on $m$, the case $m=0$ being clear by notation ($z_{j,0}^\fk=0$).

(a) We first consider the case $\fk\ne\pm\jvec$. Within this case, we first consider the case $\abs{j}\le K$, in which the strong non-resonance condition \eqref{eq-nonres-strong} holds. In the notation \eqref{eq-construction} of \eqref{eq-modsystem-coeff} we thus have $1/\abs{\al_0}\le \ga^{-2}\tau^{-2}$, $\abs{\al_1}/\abs{\al_0}\le \ga^{-1}\de^{\nu}$ since $\abs{s_{2\fk}}\le \abs{s_{\jvec+\fk}}+\abs{s_{\jvec-\fk}}$ and $\abs{\al_l}/\abs{\al_0}\le \ga^{-2}\de^{2\nu}$ for $l\ge 2$. Using the symplecticity \eqref{eq-symplectic} to replace $\psi$ by $\sinc\cdot\phi$ in \eqref{eq-modsystem-coeff}, the boundedness of $\phi$, $m(j_1,\fk^1)+m(j_2,\fk^2)\ge m(j,\fk)$ and the properties $\normv{z\cdot w}_t \le \normv{z}_t\cdot\normv{w}_t$ and $\normv{\dot{p}}_t \le \deg(p) \, \normv{p}_t$, we get 
\[
\normvbig{z_{j,m}^\fk}_t \lesssim \abs{\sinc(\tau\om_j)} \sum_{j_1+j_2\equiv j} \sum_{\fk^1+\fk^2=\fk} \sum_{m_1+m_2=m} \normvbig{z_{j_1,m_1}^{\fk^1}}_t \normvbig{z_{j_2,m_2}^{\fk^2}}_t 
\]
for the solution $z=z_{j,m}^\fk$ of \eqref{eq-construction} given by \eqref{eq-construction-offdiag}.

The same estimate also holds in the case $\fk\ne\pm\jvec$ if $\abs{j}> K$. This can be seen as follows. For these indices, only the weaker non-resonance condition \eqref{eq-nonres-weak} holds, and hence we only have $1/\abs{\al_0}\le \ga^{-2}\tau^{-2}\de^{-2\nu}$, $\abs{\al_1}/\abs{\al_0}\le \ga^{-1}$ and $\abs{\al_l}/\abs{\al_0}\le \ga^{-2}$ for $l\ge 2$. The problematic factor $\de^{-2\nu}$, however, can be compensated with the power of $\delta$ in the polynomial $p$ on the right-hand side of \eqref{eq-construction}. In fact, we have in this case
\[
m(j,\fk)=e(j)=K<e(j_1)+e(j_2)\le m(j_1,\fk^1)+m(j_2,\fk^2)
\]
since $\abs{j}>K$, and hence there is an additional factor $\de^{2\nu}$ in the polynomial $p$.

This shows that off-diagonal modulation coefficient functions $z_{j,m}^\fk$ are bounded by a constant, since only modulation coefficient functions $z_{j',m'}^{\fk'}$ with $m'<m$ can appear on right-hand side of the above estimate for  $z_{j,m}^\fk$. To get the summed estimate of the lemma, we use $\ga_j^\fk\lesssim \ga_{j_1}^{\fk^1} \ga_{j_2}^{\fk^2}$ for $\fk=\fk^1+\fk^2$ and $j\equiv j_1+j_2$, which yields
\[
\sum_{j=-M}^{M-1} \si_j \biggl( \abs{\sinc(\tau\om_j)}^{-1} \sum_{\pm\jvec\ne \fk\in\calK_j} \ga_j^\fk \normvbig{z_{j,m}^\fk}_t \biggr)^2
\lesssim \sum_{j=-M}^{M-1} \si_j \biggl( \sum_{j_1+j_2\equiv j} v_{j_1} v_{j_2} \biggr)^2
\]
with
\[
v_{j_1} = \sum_{\fk^1} \sum_{m_1=1}^{m-1} \ga_{j_1}^{\fk^1} \normvbig{z_{j_1,m_1}^{\fk^1}}_t, \qquad  
v_{j_2} = \sum_{\fk^2} \sum_{m_2=1}^{m-1} \ga_{j_2}^{\fk^2} \normvbig{z_{j_2,m_2}^{\fk^2}}_t.
\]
The first estimate then follows inductively from the algebra property \eqref{eq-algebra-basic} and $1\le \abs{\sinc(\tau\om_j)}^{-1}$.


(b) In the case $\fk=\pm\jvec$, the absolute value of the coefficient $\al_1$ is bounded from below by $\ga\tau^2\de^{2\nu}$ by the (weaker) non-resonance condition \eqref{eq-nonres-weak}. Also in this case, we have an additional factor $\de^{2\nu}$ in the polynomial on the right-hand side of \eqref{eq-construction} since $m(j,\fk)=e(j)$ and $m(j_1,\fk^1)\ge e(j)$ or $m(j_2,\fk^2)\ge e(j)$ for $\fk=\pm\jvec=\fk^1+\fk^2$. This shows that
\[
\normvbig{\dot{z}_{j,m}^{\pm\jvec}}_t \lesssim \abs{\sinc(\tau\om_j)} \sum_{j_1+j_2\equiv j} \sum_{\fk^1+\fk^2=\pm\jvec} \sum_{m_1+m_2=m} \normvbig{z_{j_1,m_1}^{\fk^1}}_t \normvbig{z_{j_2,m_2}^{\fk^2}}_t
\]
for the solution $\dot{z}=\dot{z}_{j,m}^{\pm\jvec}$ of \eqref{eq-construction} given by \eqref{eq-construction-diag}.
As in (a), this yields the second estimate of the lemma.

(c) With the definitions of $\ga_j^\fk$ and of the norm $\normv{\cdot}_t$ for $t=0$ and with the estimates $1\le\abs{\sinc(\tau\om_j)}^{-1}$ and $\abs{\sinc(\tau(\fk\cdot\fomega))}\le 1$, we get for the initial values $z_{j,m}^{\pm\jvec}(0)$ constructed with \eqref{eq-construction-init} that
\begin{align*}
\ga_j^{\pm\jvec} \absbig{z_{j,m}^{\pm\jvec}(0)} &\lesssim \sum_{\substack{\fk\in\calK_j\\ \fk\ne\jvec, \, \fk\ne-\jvec}} \abs{\sinc(\tau\om_j)}^{-1} \ga_j^\fk \absbig{z_{j,m}^\fk(0)}\\
 &\qquad +  \sum_{\fk\in\calK_j} \abs{\sinc(\tau\om_j)}^{-1} \ga_j^\fk \normvbig{\dot{z}_{j,m}^\fk}_0 + \de^{-e(j)} \om_j \abs{u_j^0} + \de^{-e(j)} \abs{\dot{u}_j^0}.
\end{align*}
With the results of (a) and (b) and with the assumption \eqref{eq-modeenergies-init} on the initial values, this yields
\[
\sum_{j=-M}^{M-1} \si_j \Bigl( \ga_j^{\pm\jvec} \absbig{z_{j,m}^{\pm\jvec}(0)} \Bigr)^2 \lesssim 1.
\]
The last estimate of the lemma is a combination of this estimate and the second estimate of the lemma proven in (b).
\end{proof}
By means of the expansion \eqref{eq-modfunc-exp} of the modulation functions $z_j^\fk$, the bounds of Lemma \ref{lemma-bounds} on the coefficients $z_{j,m}^\fk$ of this expansion imply in particular the bounds
\begin{align}
\sum_{j=-M}^{M-1} \si_j \biggl( \sum_{\fk\in\calK_j} \ga_j^\fk \de^{-m(j,\fk)} \normvbig{z_{j}^\fk}_t \biggr)^2 &\lesssim 1, \label{eq-bound-modfun-sumall}\\
\sum_{j=-M}^{M-1} \si_j \biggl( \abs{\sinc(\tau\om_j)}^{-1} \sum_{\pm\jvec\ne \fk\in\calK_j} \ga_j^\fk \de^{-m(j,\fk)} \normvbig{z_{j}^\fk}_t \biggr)^2 &\lesssim 1, \label{eq-bound-modfun-sumoffdiag}
\end{align}
on the modulation functions $z_j^\fk$ themselves. In particular, we have
\begin{equation}\label{eq-bound-modfun-easy}
z_j^\fk = \mathcal{O}(\de^{e(j)}) \myand z_j^\fk = \mathcal{O}(\de^{e(l)}) \myif k_l\ne 0
\end{equation}
by the definition of $m(j,\fk)$.

\subsection{Proof of Theorem \ref{thm-4}}\label{subsec-proof-thm4}

The proof of Theorem~\ref{thm-4} relies on the alternative form \eqref{eq-calE-alt} of the almost-invariant energies $\calE_l(t)$ of \eqref{eq-calE}.
Note again that only modulation functions $z_j^\fk$ with $k_l\ne 0$ can appear in $\calE_l(t)$ due to the factor $k_l$. By \eqref{eq-bound-modfun-easy}, this argument shows that $\calE_l(t)=\calO(\de^{2e(l)})$ has the claimed order in $\de$.

To get the precise estimate of Theorem~\ref{thm-4}, we multiply \eqref{eq-calE-alt} with $\si_l \de^{-2e(l)}$, we use $e(l)\le m(j,\fk)$ for $k_l\ne 0$, and we sum over $l$, which yields
\begin{align*}
&\sum_{l=0}^M \si_l \de^{-2e(l)} \abs{\calE_l(t)}\\
 &\qquad \le \frac{1}{2} \sum_{j=-M}^{M-1} \sum_{\fk\in\calK_j} \biggl( \sum_{l=0}^M \si_l \abs{k_l}\omega_l \biggr) \abs{\fk\cdot\fomega} \, \absbigg{\frac{\sinc(\tau(\fk\cdot\fomega))}{\sinc (\tau\omega_j)}} \, \de^{-2m(j,\fk)} \absbig{z_j^\fk\bigl(\ee^{\nu/2}t\bigr)}^2\\
 &\qquad\qquad + \frac{1}{2} \sum_{j=-M}^{M-1} \sum_{\fk\in\calK_j} \biggl( \sum_{l=0}^M \si_l \abs{k_l}\omega_l \biggr) \frac{\de^{\nu}}{\abs{\sinc (\tau\omega_j)}} \, \de^{-2m(j,\fk)} \absbig{z_{-j}^{-\fk}\bigl(\ee^{\nu/2}t\bigr)} \cdot \normvbig{\dot{z}^{\fk}_j}_t
\end{align*}
by the definition of the norm $\normv{\cdot}_t$. For the first term on the right-hand side of this estimate,
we use
\begin{equation}\label{eq-sumcalE-aux}
\sum_{l=0}^M \sigma_l |k_l| \omega_l \lesssim \sigma_j \gamma_j^\fk \myif \fk\in\calK_j
\end{equation}
and $\abs{\fk\cdot\fomega}\le \ga_j^\fk$, we use that the $\ell^1$-norm dominates the $\ell^2$-norm (more precisely, $\sum_{\fk\in\calK_j} \abs{v_j^\fk}^2 \le (\sum_{\fk\in\calK_j} \abs{v_j^\fk})^2$), and we use \eqref{eq-bound-modfun-sumall} for $\fk=\pm\jvec$ (the two sincs cancel in this case) and \eqref{eq-bound-modfun-sumoffdiag} for $\fk\ne\pm\jvec$. This shows that this first term is bounded by a constant. To show that also the second term is bounded by a constant, we apply in addition the non-resonance condition \eqref{eq-nonres-weak} to bound $\de^{\nu}/\abs{\sinc (\tau\omega_j)}$ and the Cauchy--Schwarz inequality.

\subsection{Bounds of the defect}\label{subsec-defect}

When constructing $z_j^\fk$ with the expansion \eqref{eq-modfunc-exp}, the defect $d_j^\fk$, $\fk\in\calK_j$, in \eqref{eq-modsystem} is given by
\begin{equation}\label{eq-defect}\begin{split}
d_j^\fk &= -\sum_{m=2K}^{2(2K-1)} \delta^{m(1-2\nu)} \sum_{\fk^1+\fk^2=\fk} \sum_{j_1+j_2\equiv j} \de^{2(m(j_1,\fk^1)+m(j_2,\fk^2))\nu} \\
 &\qquad\qquad\qquad\qquad \sum_{m_1+m_2=m} \phi(\tau\omega_{j_1}) z_{j_1,m_1}^{\fk^1} \phi(\tau\omega_{j_2}) z_{j_2,m_2}^{\fk^2}.
\end{split}\end{equation}
Although we consider system \eqref{eq-modsystem} defining the defect $d_j^\fk$ only for $\fk\in\calK_j$, we use this formula to define $d_j^\fk$ also for $\fk\notin\calK_j$. This will be helpful below.

\begin{lemma}\label{lemma-defect}
The defect $d_j^\fk$ in \eqref{eq-modsystem} given by \eqref{eq-defect} satisfies, for $0\le t \le 1$,
\[
\sum_{j=-M}^{M-1} \si_j \biggl( \sum_{\fk\in\Z^{M+1}} \ga_j^\fk \de^{-2m(j,\fk)\nu} \normvbig{d_j^\fk}_t \biggr)^2 \lesssim \de^{4K(1-2\nu)}
\]
with $\ga_j^\fk$ as in Lemma \ref{lemma-bounds}. 
\end{lemma}
\begin{proof}
The result follows with Lemma \ref{lemma-bounds} and the arguments used in its proof.
\end{proof}

On a higher level, the approximation $\widetilde{\fu}^n=(\widetilde{u}_{-M}^n,\dots,\widetilde{u}_{M-1}^n)^\transpose$ of \eqref{eq-mfe} to the numerical solution $\fu^n$ then has a defect $\fe^n=(e_{-M}^n,\dots,e_{M-1}^n)^\transpose$ when inserted into the numerical method \eqref{eq-method-main}:
\begin{equation}\label{eq-defect-e}
\widetilde{\fu}^{n+1} - 2\cos(\tau\fOmega) \widetilde{\fu}^n + \widetilde{\fu}^{n-1} = \tau^2 \fPsi \bigl( (\fPhi \widetilde{\fu}^n) \ast (\fPhi \widetilde{\fu}^n) \bigr) + \tau^2 \fPsi \fe^n.
\end{equation}
By construction, this defect is given by
\[
e_j^n = \sum_{\fk\in\calK_j} d_j^\fk \e^{\iu (\fk\cdot\fomega)t_n} - \sum_{\fk\in\Z^{M+1}\setminus\calK_j} \sum_{\fk^1+\fk^2=\fk} \sum_{j_1+j_2\equiv j} \phi(\tau\omega_{j_1}) z_{j_1}^{\fk^1} \phi(\tau\omega_{j_2}) z_{j_2}^{\fk^2} \e^{\iu (\fk\cdot\fomega)t_n},
\]
where the defect $d_j^\fk$ and the modulation functions are evaluated at $\de^\nu t_n$.

\begin{lemma}\label{lemma-defect-e}
The defect $\fe^n$ in \eqref{eq-defect-e} satisfies, for $0\le t_n=n\tau \le 1$,
\[
\norm{\fOmega\fe^n} \lesssim \de^{2K(1-2\nu)}
\]
with the norm $\norm{\cdot}$ of Section \ref{subsec-norms}.
\end{lemma}
\begin{proof}
We use that the set $\calK_j$ was constructed in such a way that $\fk$ belongs to $\calK_j$ if $\fk=\fk^1+\fk^2$, $j\equiv j_1+j_2$ and $m_1+m_2\le 2K-1$ for $\fk^1,\fk^2,j_1,j_2,m_1,m_2$ with $z_{j_1,m_1}^{\fk^1}\ne0$ and $z_{j_2,m_2}^{\fk^2}\ne0$. Using \eqref{eq-defect} to define $d_j^\fk$ also for $\fk\notin \calK_j$, this yields the compact form
\[
e_j^n= \sum_{\fk\in\Z^{M+1}} d_j^\fk(\de^\nu t_n) \e^{\iu (\fk\cdot\fomega)t_n}
\]
for the defect in \eqref{eq-defect-e}. From the definition of the norm $\norm{\cdot}$ (see Section \ref{subsec-norms}), we thus get
\[
\norm{\fOmega\fe^n}^2 \le \sum_{j=-M}^{M-1} \si_j \de^{-4e(j)\nu} \biggl( \sum_{\fk\in\Z^{M+1}} \om_j \absbig{d_j^\fk(\de^\nu t_n)} \biggr)^2.
\]
Using $e(j)\le m(j,\fk)$ and $\om_j\le \ga_j^\fk$, the statement of the lemma thus follows from Lemma \ref{lemma-defect}.
\end{proof}

\subsection{Bounds of the remainder}\label{subsec-remainder}

With the constructed modulation functions $z_j^\fk$, we get the approximation $\widetilde{u}_j^n$ of \eqref{eq-mfe} to the numerical solution $u_j^n$. From \eqref{align:dot-u-mfe}, we also get an approximation $\dot{\widetilde{u}}_j^n$ to $\dot{u}_j^n$. In the following lemma, we establish a bound for the approximation error.

\begin{lemma}\label{lemma-difference}
The remainders $u_j^n-\widetilde{u}_j^n$ and $\dot{u}_j^n-\dot{\widetilde{u}}_j^n$ satisfy, for $0\le t_n\le 1$,
\[
\normbig{\fOmega\klabig{\fu^n-\widetilde{\fu}^n}} + \normbig{\dot{\fu}^n-\dot{\widetilde{\fu}}^n} \lesssim \de^{2K(1-2\nu)}.
\]
\end{lemma}
\begin{proof}
(a) In a first step, we write the numerical scheme \eqref{eq-method} and its approximation by a modulated Fourier expansion in one-step form. The method in one-step form reads
\begin{equation}\label{eq-method-onestep}
\begin{pmatrix} \fu^{n+1}\\ \dot{\fu}^{n+1} \end{pmatrix}
= 
\begin{pmatrix} \cos(\tau\fOmega) & \fOmega^{-1} \sin(\tau\fOmega)\\ -\fOmega \sin(\tau\fOmega) & \cos(\tau\fOmega) \end{pmatrix}
\begin{pmatrix} \fu^{n}\\ \dot{\fu}^{n} \end{pmatrix}
+ \frac{\tau}2
\begin{pmatrix}
\tau \fPsi \fg^{n} \\ \cos(\tau\fOmega) \fPhi \fg^{n} +  \fPhi \fg^{n+1}
\end{pmatrix}
\end{equation}
with $\fg^n=(\fPhi \fu^n) \ast (\fPhi \fu^n)$, see \cite[Section XIII.2.2]{Hairer2006}.
The first line is obtained by adding \eqref{eq-method-main} and \eqref{eq-method-velocity}, and the second line is obtained by subtracting these equations with $n+1$ instead of $n$ and using the first line to replace $\fu^{n+1}$ as well as the symplecticity \eqref{eq-symplectic}. In the same way, we derive for the approximations $\widetilde{\fu}^n$ and $\dot{\widetilde{\fu}}^n$, which satisfy \eqref{eq-method-velocity} exactly (by definition \eqref{align:dot-u-mfe}) and \eqref{eq-method-main} up to a small defect given by \eqref{eq-defect-e}, 
\begin{equation}\label{eq-mfe-onestep}
\begin{pmatrix} \widetilde{\fu}^{n+1}\\ \dot{\widetilde{\fu}}^{n+1} \end{pmatrix}
= 
\begin{pmatrix} \cos(\tau\fOmega) & \fOmega^{-1} \sin(\tau\fOmega)\\ -\fOmega \sin(\tau\fOmega) & \cos(\tau\fOmega) \end{pmatrix}
\begin{pmatrix} \widetilde{\fu}^{n}\\ \dot{\widetilde{\fu}}^{n} \end{pmatrix}
+ \frac{\tau}2 
\begin{pmatrix}
\tau \fPsi \widetilde{\fg}^{n} \\ \cos(\tau\fOmega) \fPhi \widetilde{\fg}^{n} + \fPhi \widetilde{\fg}^{n+1}
\end{pmatrix}
\end{equation}
with $\widetilde{\fg}^n=\fg^n+\fe^n$. 

(b) In the following, we write
\begin{equation}\label{eq-norm-aux}
\norm{(\fv,\dot{\fv})} = \bigl( \norm{\fOmega\fv}^2 + \norm{\dot{\fv}}^2 \bigr)^{1/2}
\end{equation}
for a norm that is equivalent to the norm considered in the statement of the lemma. By induction on $n$, we prove that the numerical solutions stays small in this norm,
\[
\norm{(\fu^n,\dot{\fu}^n)} \le \bigl(\sqrt{2C_0}+t_n\bigr) \de^{1-2\nu} \myfor 0\le t_n\le 1
\]
with $C_0$ from \eqref{eq-modeenergies-init}, provided that $\de^{1-2\nu}$ is sufficiently small.
This is true for $n=0$ by the choice of the initial value \eqref{eq-modeenergies-init}. For $n>0$, we first note that, by \eqref{eq-method-onestep}, 
\[
\norm{\fOmega\fu^n}\le \norm{\fOmega\fu^{n-1}} + \norm{\dot{\fu}^{n-1}} + \tfrac12 \tau^2 \norm{\fOmega \fPsi \fg^{n-1}},
\]
and hence $\norm{\fOmega\fu^n}\le C \de^{1-2\nu}$ by induction, by \eqref{eq-symplectic}, by the boundedness of $\phi$ and by the algebra property \eqref{eq-algebra2}. 
We then observe that the matrix appearing in the one-step formulation \eqref{eq-method-onestep} of the method preserves the norm \eqref{eq-norm-aux}. Moreover, by the algebra properties \eqref{eq-algebra} and \eqref{eq-algebra2}, the second term in the one-step formulation \eqref{eq-method-onestep} can be estimated in the norm \eqref{eq-norm-aux} by $\tau\de^{2(1-2\nu)}$ up to a constant. This implies the stated bound of $\norm{(\fu^n,\dot{\fu}^n)}$.

Similarly, we get for the modulated Fourier expansion the bound
\[
\norm{(\widetilde{\fu}^n,\dot{\widetilde{\fu}}^n)} \le \bigl(\sqrt{2C_0}+t_n\bigr) \de^{1-2\nu} \myfor 0\le t_n\le 1.
\]
To get this bound, we just repeat the presented argument for the numerical solution given by \eqref{eq-method-onestep} for the modulated Fourier expansion which satisfies the very similar relation \eqref{eq-mfe-onestep}. The only difference is the additional term $\fe^n$ in \eqref{eq-mfe-onestep}, which can be estimated with Lemma \ref{lemma-defect-e}.

(c) For the difference $(\fu^n-\widetilde{\fu}^n,\dot{\fu}^n-\dot{\widetilde{\fu}}^n)$, we subtract the above one-step formulations, and then proceed as in the proof of the smallness of $(\fu^n,\dot{\fu}^n)$ and $(\widetilde{\fu}^n,\dot{\widetilde{\fu}}^n)$ in (b). We use the smallness of $(\fu^n,\dot{\fu}^n)$ and $(\widetilde{\fu}^n,\dot{\widetilde{\fu}}^n)$ together with the algebra property \eqref{eq-algebra} and Lemma \ref{lemma-defect-e} to control the difference $\fg^{n}-\widetilde{\fg}^{n}=-\fe^n$. The mentioned lemma on the defect introduces the small parameter $\de^{2K(1-2\nu)}$.
\end{proof}

With the results proven in this section, the proofs of Theorems \ref{thm-3} and \ref{thm-4} are complete.

\section{Almost-invariant energies: Proofs of Theorems~\ref{thm-5}--\ref{thm-7}}\label{sec-proof2}

\subsection{Almost-invariant energies: Proof of Theorem \ref{thm-5}}

For the proof of Theorem~\ref{thm-5}, we sum the equality \eqref{align:FastInvarianz} to get
\[
\calE_l(t_n)-\calE_l(0) = - \frac{\iu}2 \sum_{j=-M}^{M-1}\sum_{\fk\in\calK_j} \sum_{\tilde{n}=1}^{n} \tau k_l\omega_l \phi(\tau\omega_j)  z_{-j}^{-\fk}(\delta^\nu t_{\tilde{n}}) d_j^\fk(\delta^\nu t_{\tilde{n}})
\]
for the $l$th almost-invariant energy $\calE_l$.
Note again that only indices $\fk$ with $k_l\ne 0$ are relevant in the above sum, and hence we have there $z_{-j}^{-\fk}=\calO(\de^{e(l)})$ by \eqref{eq-bound-modfun-easy} and $d_j^\fk=\calO(\de^{e(l)+(2K-e(l))(1-2\nu)})$ by Lemma \ref{lemma-defect} (since $m(j,\fk)\ge e(l)$). This shows that the variation in the almost-invariant energy $\calE_l$ is indeed of the claimed order, and it shows the additional statement for $l=1$.

To get the precise estimate of Theorem~\ref{thm-5}, we proceed as in the proof of Theorem \ref{thm-4} in Section \ref{subsec-proof-thm4}: We multiply the above equation with $\si_l \de^{-2e(l)}$ and sum over $l$, we use $e(l)\le m(j,\fk)\le K$ for $k_l\ne 0$, \eqref{eq-sumcalE-aux} and $n\tau=t_n\le 1$, and we apply the Cauchy--Schwarz inequality, the estimate \eqref{eq-bound-modfun-sumall} of the modulation functions, and the estimate of the defect of Lemma \ref{lemma-defect}.

\subsection{Transitions in the almost-invariant energies: Proof of Theorem \ref{thm-6}}

We consider the situation of Theorem \ref{thm-6}, with a modulated Fourier expansion with coefficients $z_j^\fk(\de^\nu t)$, $0\le t\le 1$, constructed from $(\fu^0,\dot{\fu}^0)$ and a modulated Fourier expansion with coefficients $\widetilde{z}_j^\fk(\de^\nu t)$, $0\le t\le 1$, constructed from $(\fu^N,\dot{\fu}^N)$, where $N$ is such that $t_N=N\tau=1$. For studying the difference of the corresponding almost-invariant energies, we first consider the difference of the modulation functions themselves.

\begin{lemma}\label{lemma-interface}
Under the assumptions of Theorem \ref{thm-6}, we have
\[
\sum_{j=-M}^{M-1} \si_j \klabigg{ \sum_{\fk\in\calK_j} \gamma_j^\fk \de^{-2m(j,\fk)\nu} \normvbig{ z_j^\fk(\cdot+\de^\nu) \e^{\iu(\fk\cdot\fomega)} - \widetilde{z}_j^\fk }_0 }^2 \lesssim \de^{4K(1-2\nu)}.
\] 
\end{lemma}
\begin{proof}
Recall that the functions $z_j^\fk$ and $\widetilde{z}_j^\fk$ are constructed with the expansion \eqref{eq-modfunc-exp}. We consider here the truncated expansions
\begin{align*}
\eklabig{ z_j^\fk }^\ell(\de^\nu t) &= \de^{m(j,\fk)} \sum_{m=m(j,\fk)}^\ell \de^{(m-m(j,\fk))(1-2\nu)} z_{j,m}^\fk(\de^\nu t) ,\\
\eklabig{ \widetilde{z}_j^\fk }^\ell(\de^\nu t) &= \de^{m(j,\fk)} \sum_{m=m(j,\fk)}^\ell \de^{(m-m(j,\fk))(1-2\nu)} \widetilde{z}_{j,m}^\fk(\de^\nu t)
\end{align*}
for $\ell=1,\dots,2K-1$.
Note that these truncations coincide for $\ell=2K-1$ with $z_j^\fk$ and $\widetilde{z}_j^\fk$, respectively. We therefore study the differences
\begin{align*}
\eklabig{f_j^\fk}^\ell(\de^\nu t) & = \de^{-2m(j,\fk)\nu} \klaBig{ \eklabig{z^{\fk}_{j}}^\ell(\de^\nu t+\de^\nu) \e^{\iu (\fk\cdot\fomega)} - \eklabig{\widetilde{z}^{\fk}_{j}}^\ell(\de^\nu t) }\\
& = \sum_{m=m(j,\fk)}^{\ell}\de^{m(1-2\nu)}(z_{j,m}^{\fk}(\de^\nu t+\de^\nu) \e^{\iu (\fk\cdot\fomega)}-\tilde{z}_{j,m}^{\fk}(\de^\nu t)).
\end{align*} 

(a) We first derive equations for these differences from the defining equations \eqref{eq-modsystem-coeff} and \eqref{eq-construction-init}. For $\fk\not=\pm \langle j\rangle$ we multiply equation \eqref{eq-modsystem-coeff} by $\delta^{m(1-2\nu)}$ and sum over $m$ for $z_{j,m}^{\fk}$ and its derivatives evaluated at $\de^{\nu}t+\de^{\nu}$ and in the same way for $\widetilde{z}_{j,m}^{\fk}$ at $\delta^{\nu}t$. By multiplying the first sum with $\e^{\iu (\fk\cdot\fomega)}$, we get for the difference of these sums
\begin{align*}
&4s_{\langle j\rangle-\fk}s_{\langle j\rangle+\fk} \eklabig{f^{\fk}_{j}}^\ell + 2\iu\tau \de^\nu s_{2\fk} \eklabig{\dot{f}^{\fk}_{j}}^\ell + \tau^2 \de^{2\nu} c_{2\fk} \eklabig{\ddot{f}^{\fk}_{j}}^\ell + \ldots\\
 & = \tau^2\psi(\tau\omega_j) \sum_{\fk^1+\fk^2=\fk} \sum_{j_1+j_2\equiv j}  \sum_{m=m(j,\fk)}^{\ell}\sum_{m_{1}+m_{2}=m} \de^{m(1-2\nu)+2(m(j_1,\fk^1)+m(j_2,\fk^2)-m(j,\fk))\nu}\\
 &\qquad\qquad \phi(\tau\omega_{j_{1}})\phi(\tau\omega_{j_{2}}) \Bigl(z_{j_{1},m_{1}}^{\fk^{1}}\e^{\iu (\fk^{1}\cdot\fomega)}z_{j_{2},m_{2}}^{\fk^{2}}\e^{\iu (\fk^{2}\cdot\fomega)} - \widetilde{z}_{j_1,m_1}^{\fk^1} \widetilde{z}_{j_2,m_2}^{\fk^2}\Bigr),
\end{align*}
 where $\widetilde{z}_{j_{i},m_{i}}^{\fk^{i}}$ and $f_j^\fk$ are evaluated at $\de^{\nu}t$ and $z_{j_{i},m_{i}}^{\fk^{i}}$ are evaluated at $\de^{\nu}t+\de^\nu$ for $i=1,2$.
We use $a_1a_2-\widetilde{a}_1\widetilde{a}_2=(a_1-\widetilde{a}_1)a_2+\widetilde{a}_1(a_2-\widetilde{a}_2)$, the symmetry of the sums in $j_l$ and $\fk^l$, the convention $z_{j,m'}^\fk=0$ and $\widetilde{z}_{j,m'}^\fk=0$ for $m'<m(j,\fk)$, and $\sum_{m=1}^\ell \sum_{m_1+m_2=m} \% = \sum_{m_2=1}^{\ell-1} \sum_{m_1=1}^{\ell-m_2} \%$. This yields
\begin{align*}
&4s_{\langle j\rangle-\fk}s_{\langle j\rangle+\fk} \eklabig{f^{\fk}_{j}}^\ell + 2\iu\tau \de^\nu s_{2\fk} \eklabig{\dot{f}^{\fk}_{j}}^\ell + \tau^2 \de^{2\nu} c_{2\fk} \eklabig{\ddot{f}^{\fk}_{j}}^\ell + \ldots\\
 &\qquad = \tau^2\psi(\tau\omega_j) \sum_{\fk^1+\fk^2=\fk} \sum_{j_1+j_2\equiv j} \delta^{2(m(j_1,\fk^1)+m(j_2,\fk^2)-m(j,\fk))\nu}\\
 &\qquad\qquad \sum_{m_2=m(j_2,\fk^2)}^{\ell-1} \de^{m_2(1-2\nu)} \phi(\tau\omega_{j_{1}}) \eklabig{f^{\fk^1}_{j_1}}^{\ell-m_2} \phi(\tau\omega_{j_{2}}) \Bigl(\widetilde{z}_{j_2,m_2}^{\fk^2} + z_{j_2,m_2}^{\fk^2}\e^{\iu(\fk^{2}\cdot\fomega)} \Bigr).
\end{align*}
Using the truncated expansion $\ekla{ z_j^\fk }^\ell$ instead of ansatz \eqref{eq-modfunc-exp} we get as in Section \ref{sec-proof1} corresponding to Theorem \ref{thm-3}
\begin{align*} 
u_j^N &= \sum_{\fk\in\calK_j} \eklabig{z_j^\fk}^\ell(\de^\nu) \e^{\iu (\fk\cdot\fomega)} + \eklabig{r_j^N}^\ell\\
\dot{u}_j^N &= \bigl(\tau\sinc(\tau\omega_j)\bigr)^{-1} \sum_{\fk\in\calK_j} \Bigl( \iu s_{2\fk} \eklabig{z_j^\fk}^\ell(\de^\nu) \e^{\iu (\fk\cdot\fomega)} + \tau\de^\nu c_{2\fk} \eklabig{\dot{z}_j^\fk}^\ell(\de^\nu) \e^{\iu (\fk\cdot\fomega)}\\
 &\qquad\qquad\qquad\qquad\qquad\qquad + \tfrac{\iu}2 \tau^2\de^{2\nu} s_{2\fk}\eklabig{\ddot{z}_j^\fk}^\ell(\de^\nu) \e^{\iu (\fk\cdot\fomega)} + \dots \Bigr)
 + \eklabig{\dot{r}_j^N}^\ell,
\end{align*}
where the remainders $\ekla{r_j^N}^\ell$ and $\ekla{\dot{r}_j^N}^\ell$ of the truncated expansion satisfy
\[
\norm{\fOmega\ekla{\fr^N}^\ell} + \norm{\ekla{\dot{\fr}^N}^\ell} \lesssim \de^{(\ell+1)(1-2\nu)}.
\]
Note that the bound in Theorem \ref{thm-3} in terms of $\de$ equals $C\de^{(2K-1+1)(1-2\nu)}$ for the expansion truncated at $2K-1$. Inserting this into \eqref{eq-construction-init} for $\widetilde{z}$ (with $u_j^N$ and $\dot{u}_j^N$ instead of $u_j^0$ and $\dot{u}_j^0$) yields, for $\ell \ge e(j)$,
\begin{align*}
2\iu \om_j \eklabig{f_{j}^{\pm\jvec}}^{\ell} &= -\iu \!\! \sum_{\substack{\fk\in\calK_j\\ \fk\ne\jvec, \, \fk\ne-\jvec}} \!\! \de^{2(m(j,\fk)-e(j))\nu} \klabigg{ \om_j \pm (\fk\cdot\fomega) \, \frac{\sinc(\tau(\fk\cdot\fomega))}{\sinc(\tau\om_j)} } \eklabig{f_{j}^\fk}^\ell\\
 &\qquad \mp \sum_{\fk\in\calK_j} \frac{\de^{2(m(j,\fk)-e(j))\nu}}{\sinc(\tau\om_j)} \klaBig{ \de^\nu c_{2\fk} \eklabig{\dot{f}_{j}^\fk}^\ell + \tfrac{\iu}2 \tau \de^{2\nu} s_{2\fk} \eklabig{\ddot{f}_{j}^\fk}^\ell + \dots } \\
 &\qquad + \de^{-2e(j)\nu}\klaBig{ \iu \om_j  \eklabig{r_j^N}^\ell \pm \eklabig{\dot{r}_j^N}^\ell},
\end{align*}
where $f_j^\fk$ is evaluated at $0$.
 
(b) 
Based on the equations derived in (a), which correspond to equations \eqref{eq-modsystem-coeff} and \eqref{eq-construction-init}, it can then be shown as in the proof Lemma \ref{lemma-bounds} by induction on $\ell$ that
\[
\sum_{j=-M}^{M-1} \sigma_j \klabigg{ \sum_{\fk\in\calK_j} \gamma_j^\fk \normvbig{\eklabig{f_j^\fk}^\ell}_0 }^2 \lesssim \delta^{2(\ell+1)(1-2\nu)}.
\]
For $\ell=2K-1$ this is the bound as stated in the lemma.
\end{proof}

Now we deduce Theorem \ref{thm-6} from this lemma. We consider the difference $\calE_l(1) - \widetilde{\calE}_l(0)$ of almost-invariants in the form \eqref{eq-calE-alt}. This difference consists of differences of the form
\[
z_{-j}^{-\fk}(\de^\nu) \klabig{z_j^\fk}^{(p)}(\de^\nu)-\widetilde{z}_{-j}^{-\fk}(0) \klabig{\widetilde{z}_j^\fk}^{(p)}(0)
\]
with $p=0,1,2,\ldots$ denoting derivatives. We rewrite these differences, using
\[
a_1a_2-\widetilde{a}_1\widetilde{a}_2=(a_1\e^{-\iu(\fk\cdot\fomega)}-\widetilde{a}_1)a_2\e^{\iu(\fk\cdot\fomega)} + \widetilde{a}_1(a_2\e^{\iu(\fk\cdot\fomega)}-\widetilde{a}_2).
\]
The statement of Theorem \ref{thm-6} is then obtained by inserting ansatz \eqref{eq-modfunc-exp} for $z_j^\fk$ and $\widetilde{z}_j^\fk$ in $\sum_{l=0}^M \sigma_l \de^{-2e(l)} \absbig{ \calE_l(1) - \widetilde{\calE}_l(0) }$ and applying Lemmas \ref{lemma-bounds} and \ref{lemma-interface} together with~\eqref{eq-sumcalE-aux}.  
Note that by the previous Lemma \ref{lemma-interface} we have $z_j^\fk(\cdot+\de^\nu) \e^{\iu(\fk\cdot\fomega)} - \widetilde{z}_j^\fk = \mathcal{O}(\de^{e(l)+(2K-e(l))(1-2\nu)})$ if $k_l\ne 0$, and by \eqref{eq-bound-modfun-easy} we have $z_j^\fk=\mathcal{O}(\de^{e(l)})$ and $\widetilde{z}_j^\fk=\mathcal{O}(\de^{e(l)})$ if $k_l\ne 0$.  This yields the claimed order $\de^{2e(l)+K(1-2\nu)}$ of the transition in the $l$th almost-invariant energy, i.e., of $\calE_l(1) - \widetilde{\calE}_l(0) $, and it yields the additional statement for $l=1$. 

\subsection{Controlling mode energies by almost-invariant energies: Proof of Theorem \ref{thm-7}}

The proof of Theorem~\ref{thm-7} is done in three steps. We first show in Lemma \ref{lemma-dom-modeen} below that the goal of controlling mode energies can be achieved by controlling certain dominant terms in the modulated Fourier expansion. Then we show, in Lemma \ref{lemma-diag-dom} below, that these dominant terms can be controlled if only the diagonal ones among them are under control. Finally, we show in Lemma \ref{lemma-calE-diag} below that these diagonal dominant terms can be described essentially by the almost-invariant energies of the modulated Fourier expansion. Throughout, we assume the conditions of Theorem~\ref{thm-7} to be fulfilled. 

\begin{lemma}\label{lemma-dom-modeen}
For $0\le t\le 1$ assume that
\begin{align*}
\abs{\sinc(\tau\om_j)}^{-1} \ga_j^\fk \absbig{z_{j,e(j)}^\fk(\de^\nu t)} &\le \mathcal{C}_0 \myfor \fk\ne\pm\jvec, \, \abs{j}\le K, \, \mu(\fk)\le K,\\
\sum_{j=-M}^{M-1} \si_j \om_j^2 \absbig{z_{j,e(j)}^{\pm \jvec}(\de^\nu t)}^2 &\le \mathcal{C}_0 .
\end{align*}
Then, for $0\le t_n=n\tau \le 1$,
\[
\sum_{l=0}^M \sigma_l \delta^{-2 e(l)} E_{l}^n \le \mathcal{C}
\]
and
\[
\si_1 \de^{-2e(1)} \absBig{ E_{1}^n - \omega_{1}^{2}\Bigl(\absbig{z_{1}^{\skla{1}}(\de^\nu t_n)}^{2}+ \absbig{z_{1}^{-\skla{1}}(\de^\nu t_n)}^{2}\Bigr) } \le \mathcal{C} \de^{2(1-2\nu)}, 
\]
where $\mathcal{C}$ depends on $\mathcal{C}_0$, but is independent of $C_0$ of \eqref{eq-modeenergies-init} if $\de^{1-2\nu}$ is sufficiently small.
\end{lemma}
\begin{proof}
We consider the two terms $\tfrac12 \omega_l^2 \abs{u_{l}^n}^2$ and $\tfrac12 \abs{\dot{u}_{l}^n}^2$ in $E_l^n$ separately.

(a) For the first term, we use $\abs{u_{l}^n} \le \abs{u_{l}^n-\widetilde{u}_{l}^n} +\abs{\widetilde{u}_{l}^n}$ to get 
\[
\frac12 \sum_{l=0}^M \si_l \omega_l^2 \de^{-2e(l)} \abs{u_{l}^n}^2 \le \sum_{l=0}^M \si_l \omega_l^2 \de^{-2e(l)} \abs{u_{l}^n-\widetilde{u}_{l}^n}^2 + \sum_{l=0}^M \si_l \omega_l^2 \de^{-2e(l)} \abs{\widetilde{u}_{l}^n}^2,
\]
where the first sum is bounded by $C\de^{2K(1-2\nu)}$ corresponding to Lemma \ref{lemma-difference}. With the notation $\widetilde{u}$ of \eqref{eq-mfe} and the expansion \eqref{eq-modfunc-exp} for $z_{l}^\fk$ we find
\[
\abs{\widetilde{u}_{l}^n} \le  \sum_{\fk\in\calK_l} \de^{m(l,\fk)}\absbigg{\sum_{m=m(l,\fk)}^{2K-1}\de^{(m-m(l,\fk))(1-2\nu)}z_{l,m}^\fk(\de^\nu t_n)},
\]
where the dominant terms are $\de^{m(l,\fk)}z_{l,m(l,\fk)}^\fk$. Note that $m(l,\fk)\ge e(l)$. This yields
\[
\sum_{l=0}^M \si_l \omega_l^2 \de^{-2e(l)} \abs{\widetilde{u}_{l}^n}^2 \le \mathcal{C} + C \de^{2(1-2\nu)},
\]
with constants $\mathcal{C}$ and $C$. The constant $\mathcal{C}$ is the result of estimating the dominant terms $z_{l,m(l,\fk)}^\fk$ with $m(l,\fk)=e(l)$ using the assumption on these terms (note that $z_{l,m(l,\fk)}^\fk=0$ for $\abs{l}>K$ and $\fk\ne\pm\lvec$ by \eqref{eq-calK-case1} and \eqref{eq-calK-case2} since $m(l,\fk)=e(l)=K$ in this case). The constant $\mathcal{C}$ thus depends only on $\mathcal{C}_0$ and not on $C_0$ of \eqref{eq-modeenergies-init}. The second term $C \de^{2(1-2\nu)}$ is the result of estimating the terms $z_{l,m(l,\fk)}^\fk$ with $m(l,\fk)>e(l)$ using Lemma \ref{lemma-bounds}. The constant $C$ thus depends on $C_0$ of \eqref{eq-modeenergies-init}.

(b) For the second term, we proceed similarly. Using $\abs{\dot{u}_{l}^n} \le \abs{\dot{u}_{l}^n-\dot{\widetilde{u}}_{l}^n} + \abs{\dot{\widetilde{u}}_{l}^n}$ and again Lemma \ref{lemma-difference} to control the difference $\abs{\dot{u}_{l}^n-\dot{\widetilde{u}}_{l}^n}$, we just have to estimate $ \sum_{l=0}^M \si_l  \de^{-2e(l)} \abs{\dot{\widetilde{u}}_{l}^n}^2$.
We start from
\[
\absbig{\dot{\widetilde{u}}_{l}^n}  \le  \sum_{\fk\in\calK_l} \abs{\fk\cdot\fomega} \, \frac{\abs{\sinc(\tau(\fk\cdot\fomega))}}{\abs{\sinc(\tau\omega_l)}} \, \absbig{z_l^\fk(\de^\nu t_n)} + \abs{\sinc(\tau\omega_l)}^{-1} \sum_{\fk\in\calK_l} \normvbig{ \dot{z}_l^\fk }_{t_n},
\]
which is obtained using the velocity approximation $\dot{\widetilde{u}}$ given by \eqref{align:dot-u-mfe} and a Taylor expansion of $z_l^\fk(\ee^{\nu/2}t_{n\pm 1})$. We then insert the expansion \eqref{eq-modfunc-exp} of $z_l^\fk$ and isolate the dominant terms $\de^{m(l,\fk)} z_{l,m(l,\fk)}^\fk$. Using the assumption on the terms $z_{l,e(l)}^\fk$ and noting that $\dot{z}_{l,e(l)}^\fk=0$ by \eqref{eq-degree}, this yields similarly as in (a)
\[
\sum_{l=0}^M \si_l \de^{-2e(l)} \absbig{\dot{\widetilde{u}}_{l}^n}^2 \le \mathcal{C} + C \de^{2(1-2\nu)}.
\]
Combining this estimate with the corresponding estimate in (a), we get the estimate of the lemma, provided that $\de^{1-2\nu}$ is sufficiently small. 

(c) The additional estimate of the difference $E_{1}^n - \omega_{1}^{2}(\abs{z_{1}^{\skla{1}}(\de^\nu t_n)}^{2}+ \abs{z_{1}^{-\skla{1}}(\de^\nu t_n)}^{2})$ is obtained as follows. We first note that, as in \cite{Gauckler2012},
\begin{align*}
2E_1^{n} &= |\omega_1 u_1^{n}|^2 + |\dot{u}_1^{n}|^2 \\
&= \omega_1^2 \bigl| z_1^{\skla{1}}(\de^{\nu} t_{n})\e^{\iu\omega_1t_{n}} + z_1^{-\skla{1}}(\de^{\nu} t_{n})e^{-\iu\omega_1t_{n}} \bigr|^2  \\
&\ \ + 2\omega_1^2\mathrm{Re}\Bigl( \overline{\eta^{n}} \bigl(z_1^{\skla{1}}(\de^{\nu} t_{n})\e^{\iu\omega_1t_{n}} + z_1^{-\skla{1}}(\de^{\nu} t_{n})\e^{-\iu\omega_1t_{n}}\bigr)\Bigr) + \omega_1^2 |\eta^{n}|^2\\
 &\ \ + \omega_1^2 \bigl| \iu z_1^{\skla{1}}(\de^{\nu} t_{n})\e^{\iu\omega_1t_{n}} -\iu z_1^{-\skla{1}}(\de^{\nu} t_{n})\e^{-\iu\omega_1t_{n}} \bigr|^2\\
 &\ \ + 2\omega_1\mathrm{Re}\Bigl( \overline{\vartheta^{n}} \bigl(\iu z_1^{\skla{1}}(\de^{\nu} t_{n})\e^{\iu\omega_1t_{n}} -\iu z_1^{-\skla{1}}(\de^{\nu} t_{n})\e^{-\iu\omega_1t_{n}}\bigr)\Bigr) + |\vartheta^{n}|^2,
\end{align*}
but now with
	\begin{align*}
		\eta^{n} & = \sum_{\pm\skla{1}\ne \fk\in\calK_1} z_{1}^{\fk}(\de^{\nu}t_{n}) \e^{\iu (\fk\cdot\fomega)t_{n}} +\bigl(u_{1}^{n}-\widetilde{u}_1^n\bigr)\\
		\vartheta^{n} & = \sum_{\pm\skla{1}\ne \fk\in\calK_1} \iu(\fk\cdot \fomega) \, \frac{\sinc(\tau(\fk\cdot\fomega))}{\sinc(\tau\omega_{1})} \, z_{1}^{\fk}(\de^{\nu}t_{n})\e^{\iu (\fk\cdot\fomega)t_{n}} +\bigl(\dot{u}_{1}^{n}-\dot{\widetilde{u}}_1^n\bigr)\\
		& \hspace{2cm} + (\sinc(\tau\omega_{1}))^{-1}\sum_{\fk\in\calK_{1}}\big(c_{2\fk}\de^{\nu}\dot{z}_{1}^{\fk}(\de^{\nu}t_{n})+\dots \big) \e^{\iu (\fk\cdot\fomega)t_{n}},
	\end{align*}
which is again obtained using the velocity approximation $\dot{\widetilde{u}}$ given by \eqref{align:dot-u-mfe} for $j=1$ and a Taylor expansion of $z_1^\fk(\ee^{\nu/2}t_{n\pm 1})$. Inserting the expansion \eqref{eq-modfunc-exp} for $j=1$,  we get $\abs{\eta^{n}}\lesssim \de^{1+2(1-2\nu)}$ and $\abs{\vartheta^{n}}\lesssim \de^{1+2(1-2\nu)}$ by Lemma \ref{lemma-bounds} and Lemma \ref{lemma-difference}, since $z_{1,1}^{\fk}=z_{1,2}^{\fk}=0$ for $\fk\ne\pm\skla{1}$ and $\dot{z}_{1,1}^{\fk}=\dot{z}_{1,2}^{\fk}=0$ for all $\fk$. This yields the claimed estimate. 
\end{proof}

\begin{lemma}\label{lemma-diag-dom}
For $0\le t\le 1$ assume that
\[
\absbig{z_{j,e(j)}^{\pm\jvec}(\de^\nu t)} \le \mathcal{C}_0 \myfor \abs{j}\le K.
\]
Then,
\[
\abs{\sinc(\tau\om_j)}^{-1} \ga_j^\fk \absbig{z_{j,e(j)}^\fk(\de^\nu t)} \le \mathcal{C} \myfor \fk\ne\pm\jvec, \, \abs{j}\le K, \, \mu(\fk)\le K
\]
where $\mathcal{C}$ depends on $\mathcal{C}_0$ but is independent of $C_0$ of \eqref{eq-modeenergies-init}.
\end{lemma}
\begin{proof}
This follows from the construction of off-diagonal modulation functions as in the case of the exact solution in \cite[Lemma 4]{Gauckler2012}, see also the proof of Lemma \ref{lemma-bounds}.
\end{proof}

\begin{lemma}\label{lemma-calE-diag}
For $0\le t\le 1$ we have
\begin{align*}
\de^{-2e(0)} \absBig{ \calE_{0}(t)-\tfrac12 \omega_{0}^{2}\Bigl(\absbig{z_{0}^{\skla{0}}(\de^\nu t)}^{2}+ \absbig{z_{0}^{-\skla{0}}(\de^\nu t)}^{2}\Bigr) } &\lesssim \delta^{1-2\nu}\\
\de^{-2e(1)} \absBig{ \calE_{1}(t)-\omega_{1}^{2}\Bigl(\absbig{z_{1}^{\skla{1}}(\de^\nu t)}^{2}+ \absbig{z_{1}^{-\skla{1}}(\de^\nu t)}^{2}\Bigr) } &\lesssim \delta^{2(1-2\nu)} \\
\sum_{l=1}^M \sigma_{l} \de^{-2e(l)} \absBig{ \calE_{l}(t)-\omega_{l}^{2}\Bigl(\absbig{z_{l}^{\skla{l}}(\de^\nu t)}^{2}+ \absbig{z_{l}^{-\skla{l}}(\de^\nu t)}^{2}\Bigr) } &\lesssim \delta^{1-2\nu}.
\end{align*}
\end{lemma}
\begin{proof}
We subtract $\omega_{l}^{2}(|z_{l}^{\skla{l}}|^{2}+|z_{l}^{-\skla{l}}|^{2})$ for $l\ge 1$ and $\tfrac12 \omega_{0}^{2}(|z_{0}^{\skla{0}}|^{2}+|z_{0}^{-\skla{0}}|^{2})$ for $l=0$ from the almost-invariant energy $\calE_l(t)$ in the form \eqref{eq-calE-alt}. 
For the modulation functions that appear in this difference we still have $k_{l}\not=0$ and consequently $m(j,\fk)\geq e(l)$ and $k\not=\pm\jvec$ for $|j|\not= l$. By the expansion \eqref{eq-modfunc-exp} and Lemma \ref{lemma-bounds} we get for these functions
\begin{itemize}
\item $z_j^\fk = \mathcal{O}(\delta^{e(l)+1-2\nu})$ for $\fk\not=\pm\jvec$, since $z_{j,e(l)}^{\fk}=0$. For $\fk\ne\pm\lvec$ this follows from \eqref{eq-modfunc-exp} and \eqref{eq-degree} since $\mu(\fk)>e(l)$, and for $\fk=\pm\lvec$ from \eqref{eq-modfunc-exp} and \eqref{eq-specialzero},
\item $z_l^{\pm\lvec} = \mathcal{O}(\delta^{e(l)})$ by \eqref{eq-bound-modfun-easy} and $\dot{z}_l^{\pm\lvec} = \mathcal{O}(\delta^{e(l)+1-2\nu})$ by \eqref{eq-modfunc-exp} and \eqref{eq-degree}.
\end{itemize}
This shows that this difference is in fact of order $\mathcal{O}(\delta^{2e(l)+1-2\nu})$. The improved order for $l=1$ follows from the fact that in this case $\dot{z}_{1,2}^{\pm\skla{1}}=0$ (see Example \ref{ex-constr}), and hence $\dot{z}_1^{\pm\skla{1}} = \mathcal{O}(\delta^{e(1)+2(1-2\nu)})$. To get the summed estimate, we proceed as in the proof of Theorem \ref{thm-4} in Section \ref{subsec-proof-thm4}: We multiply the difference with $\si_l \de^{-2e(l)}$ and sum over $l$, we use \eqref{eq-sumcalE-aux}, and we apply the Cauchy--Schwarz inequality and the estimates of Lemma \ref{lemma-bounds}.
\end{proof}

Combining Lemmas \ref{lemma-dom-modeen}--\ref{lemma-calE-diag} yields the statement of Theorem~\ref{thm-7}, provided that $\de^{1-2\nu}$ is sufficiently small.

\subsection*{Acknowledgement}

We gratefully acknowledge financial support by the Deutsche Forschungsgemeinschaft (DFG) through CRC 1173, CRC 1114 and project GA 2073/2-1.

\end{document}